\documentclass[onefignum,onetabnum]{siamart171218}

\usepackage[utf8]{inputenc}
\usepackage{lipsum}
\usepackage{amsfonts}
\usepackage{graphicx}
\usepackage{epstopdf}
\usepackage{algorithmic}
\usepackage{array}
\usepackage{url}
\usepackage{amsmath,amssymb,amsbsy}
\usepackage{paralist}
\usepackage{xcolor}
\usepackage{color}
\usepackage{graphicx}
\graphicspath{{./figs/}}
\usepackage{algorithm}
\usepackage{algorithmic}
\usepackage{comment}
\usepackage{wrapfig}
\usepackage{caption}
\usepackage{subcaption}
\usepackage{bm}
\usepackage{framed}
\usepackage{fancyhdr}
\usepackage{cite}
\usepackage{cleveref}

\newtheorem{thm}{Theorem}
\newtheorem{lem}{Lemma}
\newtheorem{cor}{Corollary}
\newtheorem{prop}{Proposition}
\newtheorem{defi}{Definition}

\newcommand{\R}{\mathbb{R}}

\newcommand{\e}{\begin{equation}}
\newcommand{\ee}{\end{equation}}
\newcommand{\en}{\begin{equation*}}
\newcommand{\een}{\end{equation*}}
\newcommand{\eqn}{\begin{eqnarray}}
\newcommand{\eeqn}{\end{eqnarray}}
\newcommand{\bmat}{\begin{bmatrix}}
\newcommand{\emat}{\end{bmatrix}}
\newcommand{\btab}{\begin{tabular}}
\newcommand{\etab}{\end{tabular}}

\renewcommand{\P}[1]{\Pr\left[#1\right]}
\newcommand{\E}{\operatorname{E}}

\newcommand{\vct}[1]{\boldsymbol{#1}}
\newcommand{\mtx}[1]{\boldsymbol{#1}}

\newcommand{\T}{\mathrm{T}}

\newcommand{\trace}{\operatorname{trace}}

\newcommand{\dist}{\operatorname{dist}}

\newcommand{\set}[1]{\mathbb{#1}}

\DeclareMathOperator*{\minimize}{\text{minimize}}

\DeclareMathOperator*{\argmin}{\text{argmin}}

\def \stiefel {\operatorname*{St}}
\newcommand{\Retr}{\operatorname{Retr}}
\newcommand{\grad}{\operatorname{grad}}
\def \st {\operatorname*{subject\ to\ }}

\newcommand{\calB}{\mathcal{B}}
\newcommand{\calC}{\mathcal{C}}
\newcommand{\calD}{\mathcal{D}}

\newcommand{\calM}{\mathcal{M}}
\newcommand{\calN}{\mathcal{N}}
\newcommand{\calO}{\mathcal{O}}
\newcommand{\calP}{\mathcal{P}}

\newcommand{\calR}{\mathcal{R}}
\newcommand{\calS}{\mathcal{S}}

\newcommand{\calU}{\mathcal{U}}

\newcommand{\calW}{\mathcal{W}}
\newcommand{\calX}{\mathcal{X}}

\newcommand{\vb}{\vct{b}}

\newcommand{\vd}{\vct{d}}

\newcommand{\vo}{\vct{o}}

\newcommand{\vs}{\vct{s}}

\newcommand{\vx}{\vct{x}}
\newcommand{\vy}{\vct{y}}

\newcommand{\vxi}{\vct{\xi}}

\newcommand{\mA}{\mtx{A}}
\newcommand{\mB}{\mtx{B}}

\newcommand{\mD}{\mtx{D}}

\newcommand{\mG}{\mtx{G}}
\newcommand{\mH}{\mtx{H}}

\newcommand{\mO}{\mtx{O}}

\newcommand{\mR}{\mtx{R}}
\newcommand{\mS}{\mtx{S}}

\newcommand{\mU}{\mtx{U}}
\newcommand{\mV}{\mtx{V}}
\newcommand{\mW}{\mtx{W}}
\newcommand{\mX}{\mtx{X}}
\newcommand{\mY}{\mtx{Y}}
\newcommand{\mZ}{\mtx{Z}}
\newcommand{\mGamma}{\mtx{\Gamma}}

\newcommand{\mPhi}{\mtx{\Phi}}

\newcommand{\mSigma}{\mtx{\Sigma}}

\newcommand{\mId}{{\bm I}}

\newcommand{\setS}{\set{S}}

\graphicspath{{./figs/}}

\newlength{\imgwidth}
\newboolean{twoColVersion}
\setboolean{twoColVersion}{false}
\newcommand{\twoCol}[2]{\ifthenelse{\boolean{twoColVersion}} {#1} {#2} }

\usepackage[framemethod=tikz]{mdframed}

\ifpdf
  \DeclareGraphicsExtensions{.eps,.pdf,.png,.jpg}
\else
  \DeclareGraphicsExtensions{.eps}
\fi

\newsiamremark{remark}{Remark}
\newsiamremark{hypothesis}{Hypothesis}
\crefname{hypothesis}{Hypothesis}{Hypotheses}
\newsiamthm{claim}{Claim}

\headers{Weakly Convex Optimization over Stiefel Manifold}{X. Li, S. Chen, Z. Deng, Q. Qu, Z. Zhu, A. M.-C. So}

\title{Weakly Convex Optimization over Stiefel Manifold Using Riemannian Subgradient-Type Methods\thanks{Submitted to the editors \today. The first and second authors contributed equally to this paper. Most of the work of the first author was done when he was affiliated with the Department of Electronic Engineering, The Chinese University of Hong Kong.
		\funding{X. Li was partially supported by the University Development Fund UDF01001808 of CUHK (SZ). Q. Qu was partially supported by the Moore-Sloan fellowship.  Z. Zhu was partially supported by NSF Grant 1704458 and NSF Grant CCF-2008460. A. M.-C. So was partially supported by the Hong Kong Research Grants Council (RGC) General Research Fund (GRF) Project CUHK 14208117 and the CUHK Research Sustainability of Major RGC Funding Schemes Project 3133236.}
}
}

\author{Xiao Li\thanks{Corresponding author. School of Data Science, The Chinese University of Hong Kong, Shenzhen. 
  (\email{lixiao@cuhk.edu.cn}, \url{ https://sites.google.com/view/xli}).}
\and Shixiang Chen\thanks{Department of Industrial and Systems Engineering, Texas A$\&$M University.
		(\email{sxchen@tamu.edu}).}
\and Zengde Deng\thanks{Cainiao Network, Hangzhou, China.
		(\email{dengzengde@gmail.com}).} 
\and Qing Qu\thanks{Department of Electrical Engineering and Computer Science, University of Michigan.  
	(\email{qingqu@umich.edu}, \url{https://qingqu.engin.umich.edu/}).}
\and Zhihui Zhu\thanks{Department of Electrical and Computer Engineering, University of Denver. (\email{zhihui.zhu@du.edu}, \url{http://mysite.du.edu/\~zzhu61/}).}
\and Anthony Man-Cho So\thanks{Department of Systems Engineering and Engineering Management, The Chinese University of Hong Kong. 
		(\email{manchoso@se.cuhk.edu.hk}, \url{http://www.se.cuhk.edu.hk/\~manchoso}).}
}

\usepackage{amsopn}

\begin{document}

\maketitle

\begin{abstract}
We consider a class of nonsmooth optimization problems over the Stiefel manifold, in which the objective function is weakly convex in the ambient Euclidean space. Such problems are ubiquitous in engineering applications but still largely unexplored. We present a family of Riemannian subgradient-type methods---namely Riemannian subgradient, incremental subgradient, and stochastic subgradient methods---to solve these problems and show that they all have an iteration complexity of $\calO(\varepsilon^{-4})$ for driving a natural stationarity measure below $\varepsilon$. In addition, we establish the local linear convergence of the Riemannian subgradient and incremental subgradient methods when the problem at hand further satisfies a sharpness property and the algorithms are properly initialized and use geometrically diminishing stepsizes. To the best of our knowledge, these are the first convergence guarantees for using Riemannian subgradient-type methods to optimize a class of nonconvex nonsmooth functions over the Stiefel manifold. The fundamental ingredient in the proof of the aforementioned convergence results is a new \emph{Riemannian subgradient inequality} for restrictions of weakly convex functions on the Stiefel manifold, which could be of independent interest. We also show that our convergence results can be extended to handle a class of compact embedded submanifolds of the Euclidean space. Finally, we discuss the sharpness properties of various formulations of the robust subspace recovery and orthogonal dictionary learning problems and demonstrate the convergence performance of the algorithms on both problems via numerical simulations. 
\end{abstract}

\begin{keywords}
manifold optimization, nonconvex optimization, orthogonality constraint, iteration complexity, linear convergence, robust subspace recovery, dictionary learning
\end{keywords}

\begin{AMS}
68Q25, 65K10, 90C90, 90C26, 90C06.
\end{AMS}

\section{Introduction}\label{sec:introduction}
In this paper, we consider the problem of optimizing a function with finite-sum structure over the Stiefel manifold---i.e., 
\e\label{eq:stiefel opt problem}
\begin{split}
	&\minimize_{\mX\in \R^{n\times r}} \ f(\mX) := \frac{1}{m} \sum_{i=1}^{m} f_i(\mX)\\
	&\st \mX \in \stiefel(n,r)
\end{split}
\ee
with $\stiefel(n,r):= \{  \mX\in\R^{n\times r}: \mX^\top \mX = \mId_r \}$ and $\mId_r$ being the $r\times r$ identity matrix---where each component $f_i:\R^{n\times r} \rightarrow \R$ ($i=1,\ldots,m$) is assumed to be weakly convex in the ambient Euclidean space $\R^{n\times r}$. Recall that a function $h$ is said to be \emph{weakly convex} if $h(\cdot) + \tfrac{\tau}{2}\|\cdot\|_2^2$ is convex for some constant $\tau\ge0$~\cite{V83}. In particular, the objective function in~\eqref{eq:stiefel opt problem} can be nonconvex and nonsmooth. Our interest in~\eqref{eq:stiefel opt problem} stems from the fact that it arises in many applications from different engineering fields such as representation learning and imaging science. As an illustration, let us present two motivating applications, in which nonsmooth formulations have clear advantages over smooth ones. 



\subsection{Motivating applications}\label{subsec:motivation}


\paragraph{Application 1: Robust subspace recovery} Fitting a linear subspace to a dataset corrupted by outliers is a fundamental problem in machine learning and statistics, primarily known as robust principal component analysis (RPCA) \cite{ma2015generalized} or robust subspace recovery (RSR) \cite{lerman2018overview}. In this problem, one is given measurements $\widetilde \mY$ of the form $\widetilde \mY = \begin{bmatrix} \mY & \mO \end{bmatrix} \mGamma \in \R^{n\times m}$, where the columns of $\mY \in \R^{n\times m_1}$ form inlier points spanning a $d$-dimensional subspace $\calS$; the columns of $\mO \in \R^{n\times m_2} $ form outlier points with no linear structure; $\mGamma \in \R^{m\times m}$ is an unknown permutation, and the goal is to recover the subspace $\calS$. It is well-known that the presence of outliers can severely affect the quality of the solutions obtained by the classic PCA approach, which involves minimizing a smooth least-squares loss~\cite{ma2015generalized}. In order to obtain solutions that are more robust against outliers, the recent works \cite{lerman2015robust,lerman2018overview,maunu2019well} propose to minimize the nonsmooth least absolute deviation (LAD) loss. This leads to the formulation
\e\label{eq:RSR}
\begin{split}
	&\minimize_{\mX \in \R^{n\times d}} \ f(\mX) \;:=\;  \frac{1}{m}\sum_{i=1}^m \left\|(\mId_n - \mX \mX^\top) \widetilde\vy_i \right\|_2\\
	&\st \mX\in \stiefel(n,d),
\end{split}
\ee
where $\widetilde{\vy}_i \in \R^n$ ($i=1,\ldots,m$) denotes the $i$-th column of $\widetilde{\mY}$ and the columns of a global minimizer of \eqref{eq:RSR} are expected to form an orthonormal basis of the subspace $\calS$. The weak convexity of the components of the objective function in \eqref{eq:RSR} can be verified by following the arguments in the proof of \cite[Proposition 6]{li2018nonconvex}.  Thus, the formulation~\eqref{eq:RSR} is an instance of problem~\eqref{eq:stiefel opt problem}. On another front, the works \cite{tsakiris2018dual,zhu2018dual,zhu2019grasssub} consider a dual form of the problem, which leads to the so-called dual principal component pursuit (DPCP) formulation:
\e\label{eq:dpcp stiefel}
\begin{split}
	&\minimize_{\mX\in\R^{n\times r}} \ f(\mX) \;:= \; \frac{1}{m}\sum_{i=1}^m \left\|\widetilde \vy_i^\top \mX \right\|_2 \\
	&  \st \mX \in \stiefel(n,r).
\end{split}
\ee
In contrast to the primal formulation \eqref{eq:RSR}, the dual formulation \eqref{eq:dpcp stiefel} aims to find an orthogonal basis of $\calS^\perp$ (the orthogonal complement to $\calS$) with dimension $r=n-d$. It is clear that the components of the objective function in~\eqref{eq:dpcp stiefel} are convex, thus showing that the formulation~\eqref{eq:dpcp stiefel} is also an instance of problem~\eqref{eq:stiefel opt problem}.

\paragraph{Application 2: Learning sparsely-used dictionaries} A problem that arises in many machine learning and computer vision applications is dictionary learning (DL), whose goal is to find a suitable compact representation of certain input data $\mY = \begin{bmatrix} \vy_1, \ldots, \vy_m \end{bmatrix} \in \R^{n\times m}$~\cite{rubinstein2010dictionaries,wright2010sparse,mairal2014sparse}.  Informally, this entails factorizing the data $\mY$ into a dictionary $\mA$ and a sparse code matrix 
$\mS = \begin{bmatrix}
\vs_1,\ldots,\vs_m
\end{bmatrix}$; i.e., $\mY  \approx \mA \mS$. When the dictionary $\mA \in \R^{n\times n}$ is orthogonal and the code matrix $\mS\in\R^{n\times m}$ is sufficiently sparse, the product $\mA^\top\mY\approx\mS$ should be sparse. Thus, one may approach the problem by finding the sparsest vectors in the row space of $\mY$~\cite{spielman2012exact,qu2016finding,sun2016complete_a}. This motivates the following formulation~\cite{bai2018subgradient}:
\e\label{eq:ODL sphere}
\begin{split}
	&\minimize_{\vx\in \R^n} \ f(\vx) := \frac{1}{m} \left\| \mY^\top \vx \right\|_1 =  \frac{1}{m}  \sum_{i= 1}^{m} \left| \vy_i^\top \vx \right|\\
	&\st \vx \in \stiefel(n,1).
\end{split}
\ee
Note that the solution to \eqref{eq:ODL sphere} only returns one column of $\mA$. Thus, some extra refinement technique, such as deflation \cite{sun2016complete_b} or repetitive independent trials \cite{bai2018subgradient}, is needed to fully solve the DL problem. It has been shown in~\cite{bai2018subgradient} that under a suitable statistical model, the formulation~\eqref{eq:ODL sphere} requires fewer samples for exact recovery of the dictionary $\mA$ than the smooth variant considered in~\cite{sun2016complete_a,sun2016complete_b}. Still, since the approach based on~\eqref{eq:ODL sphere} recovers the columns of $\mA$ one at a time, it can be rather sensitive to noise. To circumvent this difficulty,   one possibility is to directly recover the orthogonal dictionary $\mA$ by
\e\label{eq:ODL orthogonal}
\begin{split}
	&\minimize_{\mX\in \R^{n \times n}} \ f(\mX) := \frac{1}{m} \left\| \mY^\top \mX \right\|_1 =  \frac{1}{m}  \sum_{i= 1}^{m} \left\| \vy_i^\top \mX \right\|_1\\
	&\st \mX \in \stiefel(n,n);
\end{split}
\ee
 cf.~\cite{wang2019unique,zhai2019complete}.
This approach can be easily extended to handle any complete (i.e., square and invertible) dictionaries via preconditioning \cite{sun2016complete_a,zhai2019complete}. Clearly, both~\eqref{eq:ODL sphere} and~\eqref{eq:ODL orthogonal} are instances of~\eqref{eq:stiefel opt problem}.

\subsection{Main contributions}

We study three \emph{Riemannian subgradient-type methods} for solving problem~\eqref{eq:stiefel opt problem}, namely Riemannian subgradient method, Riemannian incremental subgradient method, and Riemannian stochastic subgradient method (see \Cref{subsec:methods}). To analyze the convergence behavior of these methods, we first extend the surrogate stationarity measure developed in~\cite{davis2019stochasticmodel,drusvyatskiy2018efficiency} for weakly convex minimization in the Euclidean space to one for weakly convex minimization over the Stiefel manifold (see \Cref{sec:surrogate stationary measure}). Then, we show that the iterates generated by the aforementioned Riemannian subgradient-type methods will drive the surrogate stationarity measure to zero at a rate of $\calO(k^{-\frac{1}{4}})$, where $k$ is the iteration index (see \Cref{subsec:rsm,subsec:rssm}). Such a complexity guarantee matches that established in~\cite{davis2019stochasticmodel} for a host of algorithms that solve weakly convex minimization problems in the \emph{Euclidean space}. Next, we show that if problem~\eqref{eq:stiefel opt problem} further satisfies the \emph{sharpness} property (see \Cref{def:sharpness}), then the Riemannian subgradient and incremental subgradient methods with properly designed geometrically diminishing stepsizes and a good initialization will  converge to the set of local minima associated with the sharpness property at a \emph{linear} rate (see \Cref{sec:linear convergence}). To the best of our knowledge, our work is the \emph{first} to establish the iteration complexities and convergence rates of Riemannian subgradient-type methods for optimizing a class of nonconvex nonsmooth functions over the Stiefel manifold. We also extend the above convergence results to the setting where the constraint is a compact embedded submanifold of the Euclidean space (see \Cref{sec:extension}).  Lastly, we show that under certain conditions on the inlier and outlier distributions, the LAD~\eqref{eq:RSR} and DPCP~\eqref{eq:dpcp stiefel} formulations of the RSR problem satisfy the sharpness property (see \Cref{subsec:RSR}). Consequently, we are able to obtain recovery guarantees for the so-called Haystack model of the input data that are competitive with state-of-the-art results.

The key to establishing the aforementioned convergence results is an algorithm-independent property that we discovered for restrictions of weakly convex functions on the Stiefel manifold, which we term the \emph{Riemannian subgradient inequality} (see \Cref{sec:Riemannian subgradient inequality}). This is one of the main contributions of this work and could be of independent interest for other Riemannian optimization problems. We believe that our results will have broad implications on understanding the convergence behavior of algorithms for solving more general manifold optimization problems with nonsmooth objectives.

\subsection{Connections with prior arts}\label{sec:review}

\paragraph{Nonsmooth optimization in Euclidean space}

The problem of minimizing a weakly convex function over a convex constraint set is well studied in the literature. The main algorithms for this task include subgradient-type methods~\cite{davis2019stochasticmodel,davis2018subgradient,li2019incremental} and proximal point-type methods \cite{drusvyatskiy2017proximal}. The convergence analyses of these algorithms rely on a certain \emph{weakly convex inequality}. We extend this line of work by considering a nonconvex constraint set---i.e., the Stiefel manifold---and develop an analog of the weakly convex inequality on the Stiefel manifold called the \emph{Riemannian subgradient inequality}. Such an inequality allows us to resort to the analysis techniques for weakly convex minimization in the Euclidean space and prove new convergence results for our Riemannian subgradient-type methods when solving the problem of weakly convex minimization over the Stiefel manifold~\eqref{eq:stiefel opt problem}.

\paragraph{Smooth optimization over Riemannian manifold} Riemannian smooth optimization has been extensively studied over the years; see, e.g.,~\cite{absil2009optimization,jiang2017vector,Liu-So-Wu-2018,boumal2018global,hu2019brief} and the references therein.  Recently, global sublinear convergence results for Riemannian gradient descent and Riemannian trust region have been presented in \cite{boumal2018global}. The analysis relies on the assumption that the pullback of the objective function $f$ to the tangent spaces of the manifold has a Lipschitz continuous gradient, which allows one to follow the analyses of the corresponding methods for unconstrained smooth optimization. However, such an approach breaks down when $f$ is nonsmooth, as the gradient of the pullback of $f$ may not exist.

\paragraph{Nonsmooth optimization over Riemannian manifold}

In contrast to Riemannian smooth optimization, Riemannian nonsmooth optimization is relatively less explored~\cite{absil2019collection}. In the following, we briefly review some state-of-the-art results in this area and explain their limitations and connections to our results.

\emph{Riemannian nonsmooth optimization with geodesic convexity}.  Recently, the works \cite{ferreira1998subgradient,bento2017iteration,ferreira2019iteration,Sra-1st-order-geodesically-convex-2016} study the convergence behavior of Riemannian subgradient-type methods when the objective function is geodesically convex over a Riemannian manifold. Thanks to the availability of a geodesic version of the convex subgradient inequality, the conventional analysis for convex optimization in the Euclidean space can be carried over to geodesically convex optimization over a Riemannian manifold. In particular, an asymptotic convergence result is first established in \cite{ferreira1998subgradient}, while a global convergence rate of $\calO(k^{-\frac{1}{2}})$ is established in \cite{bento2017iteration,ferreira2019iteration}, for the Riemannian subgradient method. The work \cite{Sra-1st-order-geodesically-convex-2016} considers the setting where the objective function is geodesically strongly convex over the Riemannian manifold and shows that the rate can be improved to $\calO(k^{-1})$ for Riemannian projected subgradient methods. Unfortunately, these results are not useful for understanding problem \eqref{eq:stiefel opt problem}. This is because the constraint in~\eqref{eq:stiefel opt problem} is a \emph{compact} manifold, and every continuous function that is geodesically convex on a compact Riemannian manifold can only be a constant; see, e.g., \cite[Proposition 2.2]{bishop1969manifolds} and \cite{yau1974non} .   

\emph{Riemannian gradient sampling algorithms}.  For general Riemannian nonsmooth optimization, the recent works \cite{Hosseini-Uschmajew-2017,hosseini2018line} propose Riemannian gradient sampling algorithms, which are motivated by the gradient sampling algorithms for nonconvex nonsmooth optimization in the Euclidean space \cite{burke2005robust}. As introduced in \cite{Hosseini-Uschmajew-2017,hosseini2018line}, given the current iterate $\mX_k$, a typical Riemannian gradient sampling algorithm first samples some points $\big\{\mX_k^j\big\}_{j=1}^J$ in the neighborhood of $\mX_k$ at which the objective function $f$ is differentiable, where the number of sampled points $J$ usually needs to be larger than the dimension of the manifold $\calM$. Then, to obtain a descent direction, it solves the quadratic program
\e\label{subproblem:subgrad}
{\vxi}_k= -\argmin_{ \mG \in \text{conv}(\calW)} \|\mG\|^2,
\ee
where $\text{conv}(\calW)$ denotes the convex hull of   $ \calW:=\big\{\grad f(\mX_k^1),\ldots,\grad f(\mX_k^J) \big\}$ and $\grad f$ is the Riemannian gradient of $f$ on $\calM$. The update can then be performed via classical retractions on $\calM$ using the descent direction ${\vxi}_k$.  This type of algorithms can potentially be utilized to solve a large class of Riemannian nonsmooth optimization problems. However, they are only known to converge asymptotically without any rate guarantee \cite{Hosseini-Uschmajew-2017,hosseini2018line}. Moreover, in order to tackle problem \eqref{eq:stiefel opt problem} with large $n$ and $r$ using a Riemannian gradient sampling algorithm, one has to sample a large number of Riemannian gradients in each iteration, which makes the subproblem \eqref{subproblem:subgrad} very expensive to solve. By contrast, although we assume that the objective function in~\eqref{eq:stiefel opt problem} has weakly convex components, we can establish the convergence of various Riemannian subgradient-type methods with explicit rate guarantees. In addition, each iteration of those methods involves only the computation of a Riemannian subgradient, which can potentially be much cheaper.

\emph{Two types of proximal point methods}. Another classic approach to tackling Riemannian nonsmooth optimization is to apply proximal point-type methods. The idea is to iteratively compute the proximal mapping of the objective function over the Riemannian manifold \cite{Ferreira-Oliveira-PPA-Manifold-2002,de2016new}. These methods are shown to converge globally at a sublinear rate, based on the so-called sufficient decrease property. However, the main issue with this type of methods is that each subproblem is as difficult as the original problem, which renders them not practical. When specialized to the Stiefel manifold, such a difficulty has been alleviated by some recent advances \cite{chen1811proximal,huang2019riemannian,Chen2020manppa}. Specifically, they propose to compute the proximal mapping over the tangent space instead of over the Stiefel manifold, which results in a linearly constrained convex subproblem that is much easier to solve than the original problem. They also prove that the new algorithms converge globally at a sublinear rate. Nonetheless, the subproblem still needs to be solved by an iterative algorithm. By contrast, the methods considered in this paper do not need to solve expensive subproblems except for the computation of one Riemannian subgradient. As such, our overall computational complexities are much lower.

\emph{Splitting-type methods}. There are also splitting-type methods for solving Riemannian nonsmooth optimization problems, such as the manifold ADMM-type algorithms in \cite{Lai2014,Kovnatsky2016}. In this approach, the problem at hand is typically split into two subproblems---one involves optimizing a smooth function over the Riemannian manifold, the other involves optimizing a nonsmooth function without any constraint. These subproblems are then solved in an alternating manner. Despite their simplicity, these methods often do not have any convergence guarantee.

\paragraph{Nonsmooth optimization over Stiefel manifold for specific problems} Finally, we close this subsection by mentioning several problem-specific results. The recent works~\cite{bai2018subgradient} and \cite{zhu2018dual,zhu2019grasssub} propose to use the Riemannian subgradient method to solve the orthogonal DL problem \eqref{eq:ODL sphere} and RSR problem \eqref{eq:dpcp stiefel}, respectively, and establish its local linear convergence when solving these problems. The proofs are based on a certain regularity condition instead of the sharpness property studied in this work. We will give a detailed comparison between the said regularity condition and the sharpness property in \Cref{sec:linear convergence}. For now, it is worth noting that the analyses in \cite{bai2018subgradient,zhu2018dual,zhu2019grasssub} critically depend on the specific model structure of the problem at hand and cannot be easily generalized. By contrast, we develop a more general framework for analyzing Riemannian subgradient-type methods when applied to a family of nonsmooth nonconvex optimization problems over certain compact Riemannian submanifolds, which can yield both global and local convergence guarantees.

\subsection{Notation}
We use  $\T_{\mX}\stiefel :=\{\vxi \in \R^{n\times r}:  \vxi^\top\mX + \mX^\top \vxi  = 0\} $ to denote the tangent space to the Stiefel manifold $\stiefel(n,r)$ at the point $\mX\in \stiefel(n,r)$.  Let $\langle \mA,\mB \rangle = \trace(\mA^\top\mB)$ denote the Euclidean inner product of two matrices $\mA,\mB$ of the same dimensions and $\|\mA\|_F=\sqrt{\langle \mA,\mA \rangle}$ denote the Frobenius norm of $\mA$. We endow the Stiefel manifold $\stiefel(n,r)$ with the Riemannian metric inherited from the Euclidean inner product; i.e., $\langle\mX, \mY \rangle = \trace(\mX^\top \mY)$ for any $\mX,\mY\in \T_{\mZ} \stiefel$ and $\mZ\in\stiefel(n,r)$. 
For a closed set $\mathcal{C} \subseteq\R^{n\times r}$, we use $\calP_{\calC}$ to denote the orthogonal projector onto $\calC$ and $\dist(\mX,\calC):=\inf_{\mY\in\calC} \|\mX - \mY\|_F$ to denote the distance between $\mX$ and $\calC$. We use $x \lesssim y$ and $x \gtrsim y$ to denote $x \le cy$ and $x \ge cy$ for some universal constant $c$, respectively.

\section{Preliminaries}

In this section, we first review some basic notions in Riemannian optimization and then present the Riemannian subgradient-type algorithms for solving problem \eqref{eq:stiefel opt problem}.

\subsection{Optimization over Stiefel manifold}
\label{sec:opt over stiefel}

\paragraph{Riemannian subgradient and first-order optimality condition}  

By our assumption, the objective function $f$ in~\eqref{eq:stiefel opt problem} is $\tau$-weakly convex for some $\tau\ge0$; i.e, there exists a convex function $g:\R^{n\times r} \rightarrow \R$ such that $f(\mX) = g(\mX) - \frac{\tau}{2} \|\mX\|_F^2$ for any $\mX\in\R^{n\times r}$ \cite[Proposition 4.3]{V83}. Although $f$ may not be convex, we may define its (Euclidean) subdifferential $\partial f$ via 
\e\label{eq:Euclidean subdifferential}
\partial f(\mX) = \partial g(\mX) - \tau \mX, \ \ \forall \mX \in \R^{n\times r};
\ee
see \cite[Proposition 4.6]{V83}. Note that since $g$ is convex, $\partial g$ is simply its usual convex subdifferential. Hence, the subdifferential $\partial f$ in~\eqref{eq:Euclidean subdifferential} is well defined.

Using the properties of weakly convex functions in~\cite[Proposition 4.5]{V83} and the result in~\cite[Theorem 5.1]{yang2014optimality}, the Riemannian subdifferential $\partial_{\calR} f$ of $f$ on the Stiefel manifold $\stiefel(n,r)$ is given by
\e\label{eq:Riemannian subdifferential}
\partial_{\calR}f(\mX) = \calP_{\T_{\mX}\stiefel} (\partial f(\mX)), \ \ \forall \mX \in \stiefel(n,r).
\ee
In particular, given an Euclidean subgradient $\widetilde \nabla f(\mX) \in \partial f(\mX)$ of $f$ at $\mX\in\stiefel(n,r)$, we obtain a corresponding Riemannian subgradient $\widetilde \nabla_{\calR} f(\mX) \in \partial_{\calR}f(\mX)$ through $\widetilde \nabla_{\calR} f(\mX) = \calP_{\T_{\mX}\stiefel} (\widetilde \nabla f(\mX))$. Recall that for any $\mB\in \R^{n\times r}$, the projection of $\mB$ onto $\T_{\mX}\stiefel$ is given by $\calP_{\T_{\mX}\stiefel} (\mB)=  \mB -  \frac{1}{2} \mX \left(   \mB^\top  \mX  +   \mX^\top \mB    \right)$~\cite[Example 3.6.2]{absil2009optimization}. 

Using~\eqref{eq:Riemannian subdifferential}, we call $\mX \in \stiefel(n,r)$ a \emph{stationary point} of problem \eqref{eq:stiefel opt problem} if it satisfies the following first-order optimality condition:
\e\label{eq:first order optimality}
\bm{0} \in  \partial_{\calR}f(\mX).
\ee

\paragraph{Retractions on Stiefel manifold} 

To enable search along curves on the Stiefel manifold, we need the notion of a retraction (see~\cite[Definition 4.1.1]{absil2009optimization} for the definition). There are four commonly used retractions on the Stiefel manifold. These include the exponential map \cite{edelman1998geometry} and those based on the $QR$ decomposition, Cayley transformation~\cite{wen2013feasible}, and polar decomposition. It is mentioned in \cite{chen1811proximal} that among the above four retractions, the polar decomposition-based one is the most efficient in terms of computational complexity. Therefore, we shall focus on polar decomposition-based retraction, which is given by 
\e\label{eq:polar retraction}
\Retr_{\mX} (\vxi) = (\mX + \vxi ) (\mId_r + \vxi^\top\vxi)^{-\frac{1}{2}}.
\ee
However, we remark that our results also apply to the other three retractions; see \Cref{sec:extension} for a detailed discussion.

As the following lemma shows, given any $\mX \in \stiefel(n,r)$ and $\vxi\in\T_{\mX}\stiefel$, the polar decomposition-based retraction at $\mX$ essentially computes the projection of $\mX+\vxi$ onto $\stiefel(n,r)$. Moreover, this projection has a Lipschitz-like behavior, even though $\stiefel(n,r)$ is nonconvex.

\begin{lem}\label{lem:nonexpansive polar retra}
	Let $\mX\in \stiefel(n,r)$ and $\vxi\in \T_{\mX}\stiefel$ be given. Consider the point $\mX^+ = \mX + \vxi$. Then, the polar decomposition-based retraction \eqref{eq:polar retraction} satisfies $ \Retr_{\mX}(\vxi) = \mX^+\left(\mX^{+\top}\mX^+\right)^{-\frac{1}{2}}  = \calP_{\stiefel} ( \mX^+) $ and  
	\[
	\left\|\Retr_{\mX}(\vxi)- \overline \mX \right\|_F \leq \left\|  \mX^{+} - \overline \mX \right\|_F = \left\|  \mX + \vxi - \overline \mX \right\|_F, \ \ \forall \ \overline \mX\in \stiefel(n,r).
	\]
\end{lem}
\begin{proof}
	It is well known that the convex hull of the Stiefel manifold $\stiefel(n,r)$ is given by $H \equiv H(n,r):=\{ \mY\in \R^{n\times r}: \|\mY\|_2\leq 1\}$, where $\|\mY\|_2$ denotes the spectral norm (i.e. the largest singular value) of $\mY$; see, e.g.,~\cite{Journee-Nesterov-sparsePCA-JMLR-2010}. Let us first show that $ \Retr_{\mX}(\vxi) = \calP_{\stiefel} \left(\mX^{+} \right) =  \calP_{H} \left(\mX^{+} \right)$. Let $\mX^{+} = \mU\mSigma\mV^\top$ be an SVD of $\mX^+$. Since $\vxi \in \T_{\mX}\stiefel$, we have $\mX^{+\top}\mX^{+} = \mId_{r} + \vxi^\top\vxi$, which implies that all the singular values of $\mX^{+}$ are at least $1$. This, together with the Hoffman-Wielandt Theorem for singular values (see, e.g.,~\cite{stewart1990matrix}), implies that $\calP_{\stiefel} \left(\mX^{+} \right) =  \calP_{H} \left(\mX^{+} \right) = \mU\mV^\top$, as desired.
	
	Now, observe that $\Retr_{\mX}(\vxi) =  \mX^+\left(\mX^{+\top}\mX^+\right)^{-\frac{1}{2}} = \mU\mV^\top$ and $\overline \mX\in H(n,r)$. Hence, we have $\left\|\Retr_{\mX}(\vxi)- \overline \mX \right\|_F   = \left\| \calP_{H} \left(\mX^{+} \right) -  \calP_{H} \left( \overline \mX \right) \right\|_F$. Upon noting that projections onto closed convex sets are 1-Lipschitz, the proof is complete.
\end{proof}

\subsection{A family of Riemannian subgradient-type methods}\label{subsec:methods}


\paragraph{Riemannian subgradient method}
We begin by revisiting the Riemannian gradient method for smooth optimization over the Stiefel manifold.  Let $h:\R^{n\times r} \rightarrow \R$ be a smooth function and consider 
\[
\begin{array}{l}
	\displaystyle\minimize_{\mX \in \R^{n\times r}} \ h(\mX) \\
	\noalign{\smallskip}
	\st \mX\in \stiefel(n,r).
\end{array}
\]
A generic Riemannian gradient method for solving the above problem is given by
\[
\mX_{k + 1} = \Retr_{\mX_{k}} (\vxi_{k}) \quad \text{with} \quad \vxi_{k} = - \gamma_k \grad h(\mX_{k}),
\]
where $\grad h(\mX_{k})$ is the Riemannian gradient of $h$ at $\mX_k$, $\gamma_k>0$ is the stepsize, and $\Retr$ is any retraction on the Stiefel manifold; see, e.g.,~\cite[Section 4.2]{absil2009optimization}. Since problem~\eqref{eq:stiefel opt problem} involves a possibly nonsmooth objective function, one approach to tackling it is to apply a natural generalization of the Riemannian gradient method, namely the Riemannian subgradient method:
\e\label{eq:rgd update}
\boxed{
	\mX_{k + 1} = \Retr_{\mX_{k}} (\vxi_{k}) \quad \text{with} \quad \vxi_{k} = - \gamma_k \widetilde \nabla_{\calR}f(\mX_{k}).
}
\ee 
Here, recall that $\widetilde \nabla_{\calR}f(\mX_{k}) \in \partial_{\calR}f(\mX_k)$ is a Riemannian subgradient of $f$ at $\mX_k\in\stiefel(n,r)$, which can be obtained by taking $\widetilde \nabla f(\mX) \in \partial f(\mX)$ and setting $\widetilde \nabla_{\calR}f(\mX) =  \calP_{\T_{\mX}\stiefel} (\widetilde \nabla f(\mX))$; see~\Cref{sec:opt over stiefel}.

\paragraph{Riemannian incremental and stochastic subgradient methods} 
Recall that the objective function in~\eqref{eq:stiefel opt problem} has the \emph{finite-sum structure} $f = \tfrac{1}{m} \sum_{i=1}^m f_i$. In many modern machine learning tasks, the number of components $m$ can be very large. Thus, it is not desirable to evaluate the full Riemannian subgradient of $f$. This motivates us to introduce two variants of the Riemannian subgradient method \eqref{eq:rgd update}, namely the Riemannian incremental subgradient method and Riemannian stochastic subgradient method, to better exploit the finite-sum structure in~\eqref{eq:stiefel opt problem}. Let us now give a high-level description of these two methods.

Starting with the current iterate $\mX_k$, both methods generate a sequence of $m$ inner iterates $\mX_{k,1},\ldots, \mX_{k,i}, \ldots, \mX_{k,m}$ via 
\e\label{eq:rigd update}
\boxed{
	\mX_{k,i} = \Retr_{\mX_{k,i-1}} (\vxi_{k,i-1}) \quad \text{with} \quad \vxi_{k,i-1} = -\gamma_k \widetilde \nabla_{\calR} f_{\ell_i}(\mX_{k,i-1})
}
\ee 
with $\mX_{k,0}=\mX_k$, where $f_{\ell_i}$ is selected from $\{f_1,\ldots,f_m\}$ according to a certain rule. The next iterate $\mX_{k+1}$ is then obtained by setting $\mX_{k+1} = \mX_{k,m}$. The difference between the incremental and stochastic methods lies in the rule for selecting the component function $f_{\ell_i}$.  In particular, 
\begin{itemize}
	\item \emph{Riemannian incremental subgradient method} picks the component function $f_{\ell_i}$ \emph{sequentially} from $f_1$ to $f_m$---i.e., $\vxi_{k,i-1} = -\gamma_k \widetilde \nabla_{\calR} f_{i}(\mX_{k,i-1})$;
	\item \emph{Riemannian stochastic subgradient method} picks the component function $f_{\ell_i}$ \emph{independently} and \emph{uniformly} from $\{f_1, \ldots, f_m\}$ in each inner iteration \eqref{eq:rigd update}---i.e., $\vxi_{k,i-1} = -\gamma_k \widetilde \nabla_{\calR} f_{\ell_i}(\mX_{k,i-1})$ with $\ell_i \sim_{i.i.d.}\mathrm{Uniform} \big(\{1, \ldots, m\} \big) $.
\end{itemize}

\section{Riemannian Subgradient Inequality over Stiefel Manifold} \label{sec:Riemannian subgradient inequality}

Naturally, we are interested in the convergence behavior of the Riemannian subgradient-type methods introduced in \Cref{subsec:methods} when applied to problem~\eqref{eq:stiefel opt problem}. Towards that end, let us derive a useful inequality, which we call the \emph{Riemannian subgradient inequality}, for restrictions of weakly convex functions on the Stiefel manifold. The main motivation for deriving such an inequality is that an analogous one for weakly convex functions in the Euclidean space, known as the \emph{weakly convex inequality}, plays a fundamental role in the convergence analysis of subgradient-type methods for solving weakly convex minimization problems~\cite{li2018nonconvex,davis2018subgradient,davis2019stochasticmodel,li2019incremental}. To begin, recall that for a $\tau$-weakly convex function $h:\R^{n\times r} \rightarrow \R$, the weakly convex inequality states that
\e\label{eq:Euclidean weak cvx}
\begin{split}
h(\mY) &\geq h(\mX) + \left\langle  \widetilde \nabla h(\mX) , \mY - \mX \right\rangle - \frac{\tau}{2} \| \mY - \mX \|_F^2, \\
&\qquad \forall \ \widetilde \nabla h(\mX) \in \partial h(\mX); \ \mX,\mY \in \R^{n\times r}
\end{split}
\ee
\!\!\cite[Proposition 4.8]{V83}. 
The following is our extension of the above inequality to one for weakly convex functions that are restricted on the Stiefel manifold.
\begin{thm}[Riemannian Subgradient Inequality]\label{thm:Riemannian subgradient inequality}
	Suppose that $h:\R^{n\times r}\rightarrow\R$ is $\tau$-weakly convex for some $\tau\ge0$. Then, { for any bounded open convex set $\calU$ that contains $\stiefel(n,r)$, there exists a constant $L>0$ such that $h$ is $L$-Lipschitz continuous on $\calU$ and satisfies}
	\e\label{eq:Riemannian subgradient inequality weak cvx}
	\begin{split}
		h(\mY) &\geq h(\mX) + \left\langle  \widetilde \nabla_{\calR} h(\mX)  , \mY-\mX   \right\rangle  - \frac{\tau+L}{2} \|\mY - \mX\|_F^2, \\
		&\qquad \forall \ \widetilde \nabla_{\calR} h(\mX) \in \partial_{\calR} h(\mX); \mX, \mY \in \stiefel(n,r).
	\end{split}
	\ee
\end{thm}

Before we proceed to prove \Cref{thm:Riemannian subgradient inequality}, let us highlight the differences between the weakly convex inequality~\eqref{eq:Euclidean weak cvx} and the Riemannian subgradient inequality~\eqref{eq:Riemannian subgradient inequality weak cvx}. First, the former involves elements in the Euclidean subdifferential $\partial h$, while the latter involves elements in the Riemannian subdifferential $\partial_{\calR} h$. Second, the former holds for all pairs of points in the  Euclidean space $\R^{n\times r}$, while the latter only holds for all pairs of points on the Stiefel manifold $\stiefel(n,r)$. Third, the latter involves the extra compensation term $-\tfrac{L}{2} \| \mY - \mX \|_F^2$,  which accounts for the behavior of the restriction of $h$ on the Stiefel manifold $\stiefel(n,r)$.

\begin{proof}[Proof of \Cref{thm:Riemannian subgradient inequality}]
The Lipschitz continuity of $h$ on {$\calU$} follows directly from~\cite[Proposition 4.4]{V83} and {the boundedness of $\calU$}. Since $h$ is $\tau$-weakly convex on $\R^{n\times r}$, for any $\mX,\mY \in \stiefel(n,r) \subseteq \R^{n\times r}$, the inequality~\eqref{eq:Euclidean weak cvx} implies that
	\e\label{eq:cvx ineq}
	\begin{split}
		&h(\mY) \geq h(\mX) + \left\langle  \widetilde \nabla h(\mX) , \mY - \mX \right\rangle - \frac{\tau}{2} \| \mY - \mX \|_F^2 \\
		& = h(\mX) + \left\langle \calP_{\T_{\mX}\stiefel}(\widetilde \nabla h(\mX)) + \calP_{\T_{\mX}\stiefel}^{\perp} (\widetilde \nabla h(\mX)), \mY - \mX \right\rangle - \frac{\tau}{2} \| \mY - \mX \|_F^2,
	\end{split}
	\ee
	where 
	\e\label{eq:proj complement tangent stiefel}
	\calP_{\T_{\mX}\stiefel}^{\perp} (\mB)=   \frac{1}{2} \mX \left(   \mB^\top  \mX  +   \mX^\top \mB    \right), \ \ \forall \ \mB\in\R^{n\times r}
	\ee
	\!\!\cite[Example 3.6.2]{absil2009optimization}. Now, we compute
	\e \label{eq:bound weak convex term}
	\begin{split}
		\left\langle \calP_{\T_{\mX}\stiefel}^{\perp} (\widetilde \nabla h(\mX)), \mY - \mX \right\rangle   
		&= \left\langle \calP_{\T_{\mX}\stiefel}^{\perp} (\widetilde \nabla h(\mX)) ,\calP_{\T_{\mX}\stiefel}^{\perp}  (\mY - \mX) + \calP_{\T_{\mX}\stiefel} (\mY - \mX)  \right\rangle \\
		&= \left\langle \calP_{\T_{\mX}\stiefel}^{\perp} (\widetilde \nabla h(\mX)) ,  \calP_{\T_{\mX}\stiefel}^{\perp}  (\mY - \mX)   \right\rangle \\
		&= \left\langle \widetilde \nabla h(\mX) ,  \calP_{\T_{\mX}\stiefel}^{\perp}  (\mY - \mX)  \right\rangle \\
		&\stackrel{(i)}{=} \frac{1}{2}\left\langle \widetilde \nabla h(\mX) ,  \mX \left(\mY^\top\mX + \mX^\top \mY - 2\mId_{r} \right)  \right\rangle  \\
		&\stackrel{(ii)}{\geq} - \frac{1}{2} \|\widetilde \nabla h(\mX)\|_F \left\| \mY^\top\mX + \mX^\top \mY - 2\mId_{r}  \right\|_F  \\
		&\stackrel{(iii)}{\geq} - \frac{L}{2} \left\| \mY^\top\mX + \mX^\top \mY - 2\mId_{r}  \right\|_F,  
	\end{split}
	\ee
	where ($i$) comes from \eqref{eq:proj complement tangent stiefel}, ($ii$) is due to the fact that $\mX \in \stiefel(n,r)$, and { ($iii$) follows from~\cite[Theorem 9.13]{rockafellar2009variational} and the $L$-Lipschitz continuity of $h$ on $\calU$.} Note that 
	\e\label{eq:orth split}
	\left\| \mY^\top\mX + \mX^\top \mY - 2\mId_{r}  \right\|_F = \left\| (\mX-\mY)^\top  (\mX-\mY)  \right\|_F \leq \| \mX-\mY\|_F^2
	\ee
	since $\mX,\mY \in \stiefel(n,r)$.  Combining \eqref{eq:bound weak convex term} and \eqref{eq:orth split} and recalling \eqref{eq:cvx ineq}, we get
	\[
	h(\mY)  \geq h(\mX) + \left\langle \calP_{\T_{\mX}\stiefel}(\widetilde \nabla h(\mX)) , \mY - \mX \right\rangle - \frac{\tau+L}{2} \| \mY - \mX \|_F^2.
	\]
	Since $\mX,\mY\in \stiefel(n,r)$, $\widetilde \nabla h(\mX) \in \partial h(\mX)$ are arbitrary and $\partial_{\calR}h(\mX) = \calP_{\T_{\mX}\stiefel} (\partial h(\mX))$ (see \eqref{eq:Riemannian subdifferential}), the proof is complete.
\end{proof}

As we shall see in subsequent sections, the Riemannian subgradient inequality plays a similar role to the weakly convex inequality and allows us to connect the analysis of Riemannian subgradient-type methods with that of their Euclidean counterparts. In particular, equipped with \Cref{thm:Riemannian subgradient inequality}, we can obtain the iteration complexities of the Riemannian subgradient-type methods introduced in \Cref{subsec:methods} for addressing problem \eqref{eq:stiefel opt problem}. Moreover, if problem \eqref{eq:stiefel opt problem} further possesses certain sharpness property (see \Cref{def:sharpness}), then the aforementioned methods with geometrically diminishing stepsizes and a proper initialization will achieve local linear convergence to the set of so-called weak sharp minima (again, see \Cref{def:sharpness}).

Although the Riemannian subgradient inequality in \Cref{thm:Riemannian subgradient inequality} focuses on the Stiefel manifold, it can be extended to a class of compact embedded submanifolds of the Euclidean space. We shall present such an extension in \Cref{sec:extension}.

\section{Global Convergence} \label{sec:global convergence}

In this section, we study the iteration complexities of Riemannian subgradient-type methods for solving problem \eqref{eq:stiefel opt problem}. Our analysis relies on the Riemannian subgradient inequality in \Cref{thm:Riemannian subgradient inequality}. 

\subsection{Surrogate stationarity measure\label{sec:surrogate stationary measure}}
In classical Euclidean nonsmooth convex optimization, the iteration complexities of subgradient-type methods are typically presented in terms of the suboptimality gap $f(\mX_k) -\min f$; see, e.g., \cite[Theorem 3.2.2]{Nesterov:2014:ILC:2670022}, \cite[Proposition 2.3]{nedic2001convergence}. On the other hand, in Riemannian smooth optimization, which typically involves nonconvex constraints, the iteration complexities of various methods can be expressed in terms of the continuous stationarity measure $\|\text{grad} f(\mX_k)\|_F$~\cite{boumal2018global}.  However, for the Riemannian nonsmooth optimization problem \eqref{eq:stiefel opt problem}, neither the suboptimality gap $f(\mX_k) -\min f$ (due to nonconvexity) nor the minimum-norm Riemannian subgradient $\dist \left(0, \partial_{\calR}f(\mX_k) \right) $ (due to nonsmoothness) is an appropriate stationarity measure. Therefore, in order to establish the iteration complexities of Riemannian subgradient-type methods, we need to find a surrogate stationarity measure that can track the progress of those methods. 

Towards that end, we borrow ideas from the recent works~\cite{davis2019stochasticmodel,drusvyatskiy2018efficiency} on weakly convex minimization in the Euclidean space, which propose to use the gradient of the Moreau envelope of the weakly convex function at hand as a surrogate stationarity measure. To begin, let us define, for any $\lambda>0$, the following analogs of the Moreau envelope and proximal mapping for problem~\eqref{eq:stiefel opt problem}, which take into account the effect of the Stiefel manifold constraint on the problem:
\e\label{eq:Moreau envelope}
\left\{
\begin{array}{l@{\ \ \ }l}
\displaystyle f_{\lambda}(\mX) = \min_{\mY\in \stiefel(n,r)} \left\{ f(\mY) + \frac{1}{2\lambda}  \|\mY - \mX\|_F^2 \right\}, & \mX\in\stiefel(n,r),\\
\noalign{\smallskip}
\displaystyle  P_{\lambda f} (\mX) \in \argmin_{\mY\in \stiefel(n,r)} \left\{ f(\mY) + \frac{1}{2\lambda}  \|\mY - \mX\|_F^2 \right\}, & \mX\in\stiefel(n,r).
\end{array}
\right.
\ee
We remark that the Moreau envelope and proximal mapping defined above differ from those in~\cite{Ferreira-Oliveira-PPA-Manifold-2002} in that the proximal term $\mY \mapsto \tfrac{1}{2\lambda}\|\mY-\mX\|_F^2$ is based on the Euclidean distance rather than the geodesic distance. This will facilitate our later analysis.

By \eqref{eq:Riemannian subdifferential} and \eqref{eq:first order optimality}, the point $P_{\lambda f} (\mX)$ satisfies the first-order optimality condition $ \bm{0} \in \partial_{\calR} f \left(P_{\lambda f} (\mX)\right) + \tfrac{1}{\lambda} \calP_{\T_{P_{\lambda f} (\mX)}\stiefel} \left( P_{\lambda f} (\mX) - \mX\right) $. It follows that
\e \label{eq:surrogate optimality}
\boxed{
	\begin{split}
		\dist \left( \bm{0},\partial_{\calR} f \left(P_{\lambda f} (\mX)\right) \right) &\leq \lambda^{-1}\cdot \left\| \calP_{\T_{P_{\lambda f} (\mX)}\stiefel} \left( P_{\lambda f} (\mX) - \mX\right) \right\|_F \\
		&\leq \lambda^{-1} \cdot \left\|  P_{\lambda f} (\mX) - \mX \right\|_F =: \Theta(\mX).
	\end{split}
}
\ee
In particular, we see from~\eqref{eq:first order optimality} that $\mX\in\stiefel(n,r)$ is a stationary point of problem \eqref{eq:stiefel opt problem} when $\Theta(\mX) = 0$. This motivates us to use $\mX \mapsto \Theta(\mX)$ as a surrogate stationarity measure of problem \eqref{eq:stiefel opt problem} and call $\mX\in\stiefel(n,r)$ an \emph{$\varepsilon$-nearly stationary point} of problem \eqref{eq:stiefel opt problem} if it satisfies $\Theta(\mX)\leq \varepsilon$.

The careful reader may note that the proximal mapping $P_{\lambda f}$ in \eqref{eq:Moreau envelope} needs not yield a unique point at a given $\mX\in\stiefel(n,r)$. Nevertheless, for the purpose of defining the surrogate stationarity measure, we can choose any point returned by $P_{\lambda f}$ at $\mX$, as each of them plays exactly the same role in our analysis and will satisfy the convergence rate bounds in \Cref{thm:rigd global rate}. 

\subsection{Riemannian subgradient and incremental subgradient methods} \label{subsec:rsm}

Using the surrogate stationarity measure $\Theta$, we are now ready to establish the iteration complexities of the Riemannian subgradient and incremental subgradient methods. We will focus on analyzing the Riemannian incremental subgradient method, as the Riemannian subgradient method can be regarded as its special case where there is only one (i.e., $m=1$) component function. 

To begin, let us establish a relationship between the surrogate stationarity measure $\Theta$ and the sufficient decrease of the Moreau envelope $f_\lambda$.
\begin{prop}\label{prop:rigd global rate}
Suppose that each component function $f_i:\R^{n\times r}\rightarrow\R$ ($i=1,\ldots,m$) in problem~\eqref{eq:stiefel opt problem} is $\tau$-weakly convex on $\R^{n\times r}$  for some $\tau\ge0$. { Let $\mathcal{U}$ be any bounded open convex set that contains $\stiefel(n,r)$.} Furthermore, let $\{\mX_k\}$ be the sequence generated by the Riemannian incremental subgradient method \eqref{eq:rigd update} with an arbitrary initialization for solving problem \eqref{eq:stiefel opt problem}. Then, for any $\lambda < \tfrac{1}{2(L+\tau)}$ in \eqref{eq:Moreau envelope}, we have for any $k\geq 0$
	\[
	m \gamma_k  \Theta^2(\mX_k)  \leq \frac{ 2\big( f_{\lambda}\left(\mX_{k}\right)  - f_{\lambda} (\mX_{k+1}) \big)  +  \frac{\gamma_k^2 L^2}{\lambda}m^2 + \frac{\gamma_k^3L^2(L+\tau)}{\lambda} C(m) } { 2 \lambda \left(\frac{1}{2\lambda} - (L+\tau)\right) },
	\]
	where { $L>0$ is an upper bound on the Lipschitz constants of $f_1,\ldots,f_m$ on $\calU$} and $C(m) = \frac{1}{3}m(m-1)(2m-1)$.
\end{prop}
\begin{proof}
	According to  \eqref{eq:Moreau envelope}, we have
	\e\label{eq:subgradient sublinear recursion 1}
	\begin{split}
		f_{\lambda} (\mX_{k+1}) &= f\left(P_{\lambda f} (\mX_{k+1}) \right) + \frac{1}{2\lambda}  \left\| P_{\lambda f} (\mX_{k+1}) - \mX_{k+1} \right\|_F^2\\
		&\leq f\left(P_{\lambda f} (\mX_{k})\right) + \frac{1}{2\lambda}  \left\| P_{\lambda f} (\mX_{k}) - \mX_{k+1} \right\|_F^2,
	\end{split}
	\ee
	where the last inequality follows from the optimality of $P_{\lambda f} (\mX_{k+1})$ and the fact that $P_{\lambda f} (\mX_{k})\in \stiefel(n,r)$. We claim that for $l=1,\ldots,m$,
	\e\label{eq:iterates to prox mapping dist}
	\begin{split}
		\left\| P_{\lambda f} (\mX_{k}) - \mX_{k,l} \right\|_F^2 & \leq \left\|  \mX_{k} - P_{\lambda f} (\mX_{k})  \right\|_F^2 - 2\gamma_k \sum_{i=1}^{l} \left(  f_i(\mX_{k,i-1}) - f_i\left( P_{\lambda f} (\mX_{k}) \right)  \right) \\
		&\quad + \gamma_k (L+\tau) \sum_{i=1}^{l} \left\|  \mX_{k,i-1} - P_{\lambda f} (\mX_{k})  \right\|_F^2 + l\gamma_k^2 L^2.
	\end{split}
	\ee
	The proof is by induction on $l$. For $l=1$, recalling that $\mX_{k,0} = \mX_k$, we compute
	\e\label{eq:first induction recursion}
	\begin{split}
		\left\| P_{\lambda f} (\mX_{k}) - \mX_{k,1} \right\|_F^2 & \leq \left\|  \mX_{k} +\vxi_{k,0} - P_{\lambda f} (\mX_{k})  \right\|_F^2 \\
		&\leq  \left\|  \mX_{k}  - P_{\lambda f} (\mX_{k})  \right\|_F^2     - 2\gamma_k \left(  f_1(\mX_{k}) - f_1\left( P_{\lambda f} (\mX_{k}) \right)  \right) \\
		&\quad + \gamma_k (L+\tau)  \left\|  \mX_{k} - P_{\lambda f} (\mX_{k})  \right\|_F^2 + \gamma_k^2 L^2,
	\end{split}
	\ee
	where we used~\eqref{eq:rigd update} and \Cref{lem:nonexpansive polar retra} in the first inequality and \Cref{thm:Riemannian subgradient inequality} and the fact that $\left\| \widetilde \nabla_{\calR} f_i(\mX_{k,i-1})  \right\|_F  \leq  \left\| \widetilde \nabla f_i(\mX_{k,i-1})  \right\|_F   \leq L$ in the second inequality.  The inductive step can be completed by following the same derivations as in \eqref{eq:first induction recursion}. Thus, the claim \eqref{eq:iterates to prox mapping dist} is established. Setting $l =m$ in \eqref{eq:iterates to prox mapping dist} and plugging it into \eqref{eq:subgradient sublinear recursion 1}, we obtain
	\e\label{eq:subgradient sublinear recursion 2}
	\begin{split}
		f_{\lambda} (\mX_{k+1}) &\leq f_{\lambda}\left(\mX_{k}\right)  + \frac{\gamma_k}{\lambda} \sum_{i=1}^{m} \left( f_i\left( P_{\lambda f} (\mX_{k}) \right)  -  f_i(\mX_{k,i-1})  \right) \\
		&\quad + \frac{\gamma_k (L+\tau)}{2\lambda} \sum_{i=1}^{m} \left\|  \mX_{k,i-1} - P_{\lambda f} (\mX_{k})  \right\|_F^2 + \frac{m\gamma_k^2 L^2}{2\lambda},
	\end{split}
	\ee
	where we used the relation $f_{\lambda}\left(\mX_{k}\right) = f\left(P_{\lambda f} (\mX_{k}) \right) + \frac{1}{2\lambda} \left\|  \mX_{k} - P_{\lambda f} (\mX_{k})  \right\|_F^2$ (since $\mX_{k}\in \stiefel(n,r)$).  
	
	Next, we claim that for $i=1,\ldots,m$,
	\e\label{eq:inner iterates distance sitefel}
	\left\| \mX_{k,i-1} - \mX_{k} \right\|_F \leq (i-1) \gamma_k L.
	\ee
	The proof is again by induction on $i$. The claim trivially holds when $i = 1$. Suppose that~\eqref{eq:inner iterates distance sitefel} holds for $i = j$. For $i=j+1$, we compute $ \left\| \mX_{k,j} - \mX_{k} \right\|_F  \leq \| \mX_{k,j-1} + \vxi_{k,j-1} - \mX_k\|_F \leq j \gamma_k L $, where we used~\eqref{eq:rigd update} and \Cref{lem:nonexpansive polar retra} in the first inequality. This completes the inductive step and the proof of the claim.	
	
	With \eqref{eq:inner iterates distance sitefel}, we have
	\e\label{eq:function value bound}
	\begin{split}
		f_i\left( P_{\lambda f} (\mX_{k}) \right)  -  f_i(\mX_{k,i-1}) &=  f_i\left( P_{\lambda f} (\mX_{k}) \right) -f_i(\mX_{k})  + f_i(\mX_{k})  -  f_i(\mX_{k,i-1}) \\
		& \leq  (i-1) \gamma_kL^2 +  f_i\left( P_{\lambda f} (\mX_{k}) \right) -  f_i( \mX_{k})
	\end{split}
	\ee
	and 
	\e\label{eq:iterates bound}
	\begin{split}
		\left\|  \mX_{k,i-1} - P_{\lambda f} (\mX_{k})  \right\|_F^2 &= \left\|  \mX_{k,i-1} -\mX_{k} + \mX_{k} - P_{\lambda f} (\mX_{k})  \right\|_F^2  \\
		& \leq 2(i-1)^2 \gamma_k^2 L^2 + 2 \left\|  \mX_{k} - P_{\lambda f} (\mX_{k})  \right\|_F^2.
	\end{split}
	\ee
	Plugging \eqref{eq:function value bound} and \eqref{eq:iterates bound} into  \eqref{eq:subgradient sublinear recursion 2} yields
	\e\label{eq:subgradient sublinear recursion 3}
	\begin{split}
		f_{\lambda} (\mX_{k+1}) &\leq f_{\lambda}\left(\mX_{k}\right)   +  \frac{m^2 \gamma_k^2 L^2 + \frac{1}{3}m(m-1)(2m-1)\gamma_k^3L^2(L+\tau)}{2\lambda}  \\
		&\quad + \frac{m \gamma_k}{\lambda} \left( f\left( P_{\lambda f} (\mX_{k}) \right)  -  f(\mX_{k})  \right)  + \frac{m \gamma_k (L+\tau)}{\lambda} \left\|  \mX_{k} - P_{\lambda f} (\mX_{k})  \right\|_F^2.
	\end{split}
	\ee
	By definition of the Moreau envelope and proximal mapping  in \eqref{eq:Moreau envelope}, we have
	\e\label{eq:subgradient sublinear recursion 4}
	\begin{split}
		& -\Big[ f(\mX_{k}) - f\left( P_{\lambda f} (\mX_{k}) \right) - (L+\tau) \left\|  \mX_{k} - P_{\lambda f} (\mX_{k})  \right\|_F^2 \Big] \\
		=& -\Bigg[ f(\mX_{k}) - \Big( f\left( P_{\lambda f} (\mX_{k}) \right) + \frac{1}{2\lambda} \left\|  \mX_{k} - P_{\lambda f} (\mX_{k})  \right\|_F^2   \Big)  \\
		& \qquad+ \left(\frac{1}{2\lambda} - (L+\tau) \right)  \left\|  \mX_{k} - P_{\lambda f} (\mX_{k})  \right\|_F^2 \Bigg] \\
		\leq& - \left(\frac{1}{2\lambda} - (L+\tau)\right)  \left\|  \mX_{k} - P_{\lambda f} (\mX_{k})  \right\|_F^2,
	\end{split}
	\ee
	where the last inequality is due to $f_\lambda (\mX_{k}) = f\left( P_{\lambda f} (\mX_{k}) \right) + \frac{1}{2\lambda} \left\|  \mX_{k} - P_{\lambda f} (\mX_{k})  \right\|_F^2$ (since $\mX_{k}\in \stiefel(n,r)$) and $f_\lambda (\mX_{k}) \leq f (\mX_{k})$. Since $\lambda < \tfrac{1}{2(L+\tau)}$ by assumption, the desired result then follows by substituting \eqref{eq:subgradient sublinear recursion 4} into \eqref{eq:subgradient sublinear recursion 3}  and recognizing that $\Theta(\mX_k) = \lambda^{-1}  \left\|   P_{\lambda f} (\mX_{k})  - \mX_{k} \right\|_F$ (see \eqref{eq:surrogate optimality}).
\end{proof}

Using \Cref{prop:rigd global rate}, we obtain our iteration complexity result for the Riemannian subgradient and incremental subgradient methods.

\begin{thm}\label{thm:rigd global rate}
Under the setting of \Cref{prop:rigd global rate}, the following hold:
	\begin{enumerate}[(a)]
		\item If we choose the constant stepsize $\gamma_k = \tfrac{1}{m\sqrt{T+1}}$, $k=0,1,\ldots$ with $T$ being the total number of  iterations, then
		\[
		\begin{aligned}
		\min_{0\leq k\leq T} \Theta^2(\mX_k)   \leq \frac{ 2\big( f_{\lambda}\left(\mX_{0}\right)  - \min f_{\lambda} \big)  +   \frac{L^2}{\lambda}   + \frac{L^2 (L+\tau)}{\lambda m^3 } C(m) }{2\lambda \left(\frac{1}{2\lambda} - (L+\tau)\right) \sqrt{T+1}}. 
		\end{aligned}
		\]
		\item If we choose the diminishing stepsizes $\gamma_k = \tfrac{1}{m\sqrt{k+1}}$, $k=0,1,\ldots$, then 
		\[
		\begin{aligned}
		\min_{0\leq k\leq T} \Theta^2(\mX_k)    \leq    \frac{2 \big(f_{\lambda}\left(\mX_{0}\right)  - \min f_{\lambda}  \big)  +  \left( \frac{L^2}{\lambda}    + \frac{L^2 (L+\tau)}{\lambda m^3 } C(m)  \right) \big(\ln(T+1) +1 \big) }{2 \lambda \left(\frac{1}{2\lambda} - (L+\tau)\right) \sqrt{T+1}}.
		\end{aligned}
		\]
	\end{enumerate}
\end{thm}
\begin{proof}
By summing both sides of the relation in \Cref{prop:rigd global rate} over $k=0,1,\ldots,T$, we deduce that
	\[
		\min_{0\leq k\leq T} \Theta^2(\mX_k)    \leq   \frac{ 2\big( f_{\lambda}\left(\mX_{0}\right)  - \min f_{\lambda} \big)  +    \frac{L^2}{\lambda} m^2\sum_{k=0}^{T} \gamma_k^2  + \frac{L^2(L+\tau)} {\lambda} C(m)  \sum_{k=0}^{T} \gamma_k^3 }{2\lambda \left(\frac{1}{2\lambda} - (L+\tau)\right) m \sum_{k=0}^{T} \gamma_k}.
	\]
	The result in (a) follows immediately by substituting $\gamma_k = \tfrac{1}{m\sqrt{T+1}}$ into the above inequality, while that in (b) follows by substituting $\gamma_k = \frac{1}{m\sqrt{k+1}}$ into the above inequality and noting that $\sum_{k=0}^{T} \tfrac{1}{\sqrt{k+1}} > \sqrt{T+1}$ and $\sum_{k=0}^{T} {\tfrac{1}{k+1}} < \ln(T+1) + 1$.
\end{proof}

By taking $\lambda = \tfrac{1}{4 (L+\tau)}$ and using the constant stepsize $\gamma_k = \tfrac{1}{m \sqrt{T+1}}$, $k=0,1,\ldots$, we see from \Cref{thm:rigd global rate} that 
\[
\min_{0\leq k\leq T} \Theta(\mX_k)    \leq  \frac{ 2 \sqrt{ \big(f_{\lambda}\left(\mX_{0}\right)  - \min f_{\lambda}  \big)  +  2 L^2 (L+\tau)(1+L+\tau)  }  }{(T+1)^{1/4}}.
\]
In particular, the iteration complexity of the Riemannian (incremental) subgradient method for computing an $\varepsilon$-nearly stationary point of problem~\eqref{eq:stiefel opt problem} is $\calO(\varepsilon^{-4})$. It is worth noting that this matches the iteration complexity of a host of methods for solving weakly convex minimization problems in the Euclidean space~\cite{davis2019stochasticmodel}.


\subsection{Riemannian stochastic subgradient method} \label{subsec:rssm}

Now, let us turn to analyze the Riemannian stochastic subgradient method. Instead of focusing on objective functions with a finite-sum structure as in \eqref{eq:stiefel opt problem}, we consider the following more general stochastic optimization problem over the Stiefel manifold:
\e\label{eq:stiefel stochastic opt problem}
\begin{split}
	&\minimize_{\mX \in \R^{n\times r}}  \ f(\mX) := \E_{\zeta \sim D } [g(\mX,\zeta)] \\
	&\st \mX \in \stiefel(n,r).
\end{split}
\ee
Here, we assume that the function $\mX\mapsto g(\mX,\zeta)$ is $\tau$-weakly convex ($\tau\ge0$) for each realization $\zeta$ and the function $f$ is finite-valued on $\R^{n \times r}$. Furthermore, we assume the existence of a bounded open convex set $\calU$ containing $\stiefel(n,r)$ such that $\mX\mapsto g(\mX,\zeta)$ is Lipschitz continuous on $\calU$ with some constant $L(\zeta)>0$ and $L^2 = \E_{\zeta \sim D }[ L(\zeta)^2 ] < +\infty$. This would then imply that for any $\mX \in \stiefel(n,r)$, we have
\e \label{eq:g-sub-bd}
\E_{\zeta \sim D }\left[ \left\|\widetilde \nabla g(\mX, \zeta) \right\|_F^2 \right] \le L^2,
\ee
where $\widetilde \nabla g(\mX, \zeta) \in \partial g(\mX,\zeta)$. Moreover, the function $f$ is $L$-Lipschitz continuous on $\calU$. 
When $D$ is the empirical distribution on $m$ data samples, problem \eqref{eq:stiefel stochastic opt problem} reduces to our original finite-sum optimization problem \eqref{eq:stiefel opt problem}. If all the component functions are finite-valued and weakly convex, then the above two assumptions hold.

Now, suppose that the Riemannian stochastic subgradient method is equipped with a \emph{Riemannian stochastic subgradient oracle}, which has the following properties:
\begin{enumerate}[(a)]
	\item The oracle can generate i.i.d. samples according to the distribution $D$.
	\item Given a point $\mX\in\stiefel(n,r)$, the oracle generates a sample $\zeta\sim D$ and returns a stochastic subgradient $\widetilde \nabla g(\mX, \zeta) \in \partial g(\mX,\zeta)$ with $\E_{\zeta \sim D } [\widetilde \nabla g(\mX, \zeta)] \in \partial f(\mX)$, from which one can obtain a Riemannian stochastic subgradient $\widetilde \nabla_{\calR} g(\mX, \zeta) \in \partial_{\calR} g(\mX,\zeta)$ with $\E_{\zeta \sim D } [\widetilde \nabla_{\calR} g(\mX, \zeta)] \in \partial_{\calR} f(\mX)$.

\end{enumerate}
We remark that the above properties mirror those of the stochastic subgradient oracle for stochastic optimization in the Euclidean space; see, e.g., Assumptions (A1) and (A2) in \cite{nemirovski2009robust}. 

At the current iterate $\mX_k$, the Riemannian stochastic subgradient oracle generates a sample $\zeta_k \sim D$ that is independent of $\{\zeta_0,\ldots, \zeta_{k-1}\}$ and returns a Riemannian stochastic subgradient $\widetilde \nabla_{\calR} g(\mX_k, \zeta_k)$. Then, the Riemannian stochastic subgradient method generates the next iterate $\mX_{k+1}$ via
\e\label{eq:sgd update}
\mX_{k+1} = \Retr_{\mX_{k}} (\vxi_{k}) \quad\text{with}\quad \vxi_{k} = - \gamma_k \widetilde \nabla_{\calR} g(\mX_{k}, \zeta_k).
\ee
This generalizes the update~\eqref{eq:rigd update} introduced in \Cref{subsec:methods} for the case where $D$ is the empirical distribution on $m$ data samples.


Similar to the analysis of the Riemannian subgradient and incremental subgradient methods, we begin by establishing the following result; cf.~\Cref{prop:rigd global rate}:
\begin{prop}\label{prop:stochastic subgradient global rate}
Suppose that the aforementioned assumptions on problem~\eqref{eq:stiefel stochastic opt problem} hold, and that a Riemannian stochastic subgradient oracle having properties (a)--(b) above is available. Let $\{\mX_k\}$ be the sequence generated by the Riemannian stochastic subgradient method \eqref{eq:sgd update} with arbitrary initialization for solving problem \eqref{eq:stiefel stochastic opt problem}. Then, for any $\lambda < \frac{1}{L+\tau}$ in \eqref{eq:Moreau envelope}, we have 
	\[
	\gamma_k \E\left[ \Theta^2(\mX_k) \right]  \leq  \frac{ 2 \big( \E\left[ f_{\lambda} (\mX_{k}) \right] - \E \left[ f_{\lambda} (\mX_{k+1}) \right] \big) + \frac{\gamma_k^2 L^2}{\lambda} }  { \lambda  \left( \frac{1}{\lambda} - (L+\tau)\right)}, \ \ \forall \ k\geq 0.
	\]
\end{prop}
\begin{proof}
Using~\eqref{eq:Moreau envelope}, the optimality of $P_{\lambda f} (\mX_{k+1})$,  \Cref{lem:nonexpansive polar retra}, and the fact that $P_{\lambda f} (\mX_{k}) \in \stiefel(n,r)$, we obtain
	\begin{align*}
		&\E_{\zeta_k\sim D} \left[ f_{\lambda} (\mX_{k+1}) \right] 
		\leq  f\left(P_{\lambda f} (\mX_{k})\right) + \frac{1}{2\lambda}  \E_{\zeta_k \sim D}\left[ \left\| P_{\lambda f} (\mX_{k}) - \mX_{k+1} \right\|_F^2 \right] \\
		\leq& \ f\left(P_{\lambda f} (\mX_{k})\right) + \frac{1}{2\lambda}  \E_{\zeta_k \sim D}\left[ \left\|  \mX_{k}  - \gamma_k \widetilde \nabla_{\calR} g(\mX_k, \zeta_k) -  P_{\lambda f} (\mX_{k}) \right\|_F^2 \right] \\
		\leq& \ f_{\lambda} (\mX_{k}) + \frac{\gamma_k}{\lambda} \E_{\zeta_k \sim D} \left[ \left\langle \widetilde \nabla_{\calR} g(\mX_k, \zeta_k),   P_{\lambda f} (\mX_{k}) - \mX_{k} \right\rangle \right] + \frac{\gamma_k^2	L^2}{2\lambda},
	\end{align*}
	where the first inequality is due to the optimality of $P_{\lambda f} (\mX_{k+1})$, the second inequality comes from \Cref{lem:nonexpansive polar retra} and the fact that $P_{\lambda f} (\mX_{k}) \in \stiefel(n,r)$, and the third inequality is due to~\eqref{eq:g-sub-bd} and the fact that $\left\| \widetilde \nabla_{\calR} g(\mX, \zeta) \right\|_F \le \left\|\widetilde \nabla g(\mX, \zeta) \right\|_F$. Since we have $\E_{\zeta_k \sim D} \left[ \widetilde \nabla_{\calR} g(\mX_k, \zeta_k) \right] \in \partial_{\calR} f(\mX_k)$, the $L$-Lipschitz continuity of $f$ on $\calU$, \Cref{thm:Riemannian subgradient inequality}, and \eqref{eq:subgradient sublinear recursion 4} imply that
	\begin{align*}
		 &\E_{\zeta_k\sim D} \left[ f_{\lambda} (\mX_{k+1}) \right] 
		\leq  f_{\lambda} (\mX_{k}) - \frac{\gamma_k}{2\lambda}  \left( \frac{1}{\lambda} - (L+\tau) \right)   \left\|   \mX_k - P_{\lambda f} (\mX_{k}) \right\|_F^2 + \frac{\gamma_k^2	L^2}{2\lambda}.
	\end{align*}
	 Upon taking expectation with respect to all the previous realizations $\zeta_0, \ldots, \zeta_{k-1}$ on both sides, we get
	\begin{align*}
		\E \left[ f_{\lambda} (\mX_{k+1}) \right] &\leq  \E\left[ f_{\lambda} (\mX_{k}) \right] - \frac{\gamma_k}{2\lambda}  \left( \frac{1}{\lambda} - (L+\tau)  \right)   \E\left[ \left\| P_{\lambda f} (\mX_{k})  -   \mX_k \right\|_F^2 \right] + \frac{\gamma_k^2	L^2}{2\lambda}.
	\end{align*}
 The desired result then follows by rearranging the above inequality and recognizing that $\Theta(\mX_k) = \lambda^{-1}  \left\|   P_{\lambda f} (\mX_{k})  - \mX_{k} \right\|_F$ (see \eqref{eq:surrogate optimality}).  
\end{proof}

Now, we can bound the iteration complexity of the Riemannian stochastic subgradient method using \Cref{prop:stochastic subgradient global rate}.

\begin{thm}\label{thm:stochastic subgradient global rate}
Under the setting of \Cref{prop:stochastic subgradient global rate}, suppose that we choose the constant stepsize $\gamma_k = \tfrac{1}{\sqrt{T+1}}$, $k=0,1,\ldots$ with $T$ being the total number of iterations and the algorithm returns $\mX_{\overline k}$ with $\overline k$ sampled from $\{1,\ldots,T\}$ uniformly at random. Then,  we have 
	\[
	\E\left[ \Theta^2\left(\mX_{ \overline k}\right) \right] \leq \frac{1} { \lambda \left( \frac{1}{\lambda} - (L+\tau) \right) }  \frac{   2 \big( f_{\lambda} (\mX_{0})   - \min f_{\lambda} \big) + \frac{L^2}{\lambda} }   {\sqrt{T+1}},
	\]
	where the expectation is taken over all random choices by the algorithm.
\end{thm}
\begin{proof}
By summing both sides of the relation in \Cref{prop:stochastic subgradient global rate} over $k=0,1,\ldots,T$, we have
	\[ 
	\sum_{k=0}^T \gamma_k   \E\left[\Theta^2(\mX_k) \right]      \leq  \frac{ 2\big( f_{\lambda} (\mX_{0})  - \min f_{\lambda} \big)  + \frac{L^2}{\lambda} \sum_{k=0}^T \gamma_k^2 }{\lambda  \left( \frac{1}{\lambda} - (L+\tau) \right) }.
	\] 
	It follows that
	\[ 
		\sum_{k=0}^T \frac{\gamma_k}{\sum_{k=0}^T \gamma_k}    \E\left[\Theta^2(\mX_k) \right]   \leq   \frac{1} {\lambda \left( \frac{1}{\lambda} - (L+\tau) \right) }  \frac{ 2\big(  f_{\lambda} (\mX_{0})  - \min f_{\lambda} \big) + \frac{	L^2}{\lambda} \sum_{k=0}^T \gamma_k^2 }   {\sum_{k=0}^T \gamma_k}.
	\] 
	To complete the proof, it remains to substitute $\gamma_k = \tfrac{1}{\sqrt{T+1}}$ into the above inequality and note that the resulting LHS is exactly $ \E\left[ \Theta^2\left(\mX_{ \overline k}\right)  \right]$ with the expectation being taken with respect to $\zeta_0,\ldots,\zeta_{T-1},\overline k$. 
\end{proof}

\section{Local Linear Convergence for Sharp Instances}\label{sec:linear convergence}

So far our discussion on problem~\eqref{eq:stiefel opt problem} does not assume any structure on the objective function $f$ besides weak convexity. However, many applications, such as those discussed in \Cref{subsec:motivation}, give rise to weakly convex objective functions that are not arbitrary but have rather concrete structure. It is thus natural to ask whether the methods we considered can exploit  this structure and provably achieve faster convergence rates than those established in \Cref{sec:global convergence}. In this section, we introduce a regularity property of problem~\eqref{eq:stiefel opt problem} called \emph{sharpness} and show that the Riemannian subgradient and incremental subgradient methods will achieve a local linear convergence rate when applied to instances of~\eqref{eq:stiefel opt problem} that possess the sharpness property. Then, we will discuss in \Cref{sec:application} how the notion of sharpness captures, in a unified manner, the structure of both the  dual principal component pursuit (DPCP) formulation \eqref{eq:dpcp stiefel} of the robust subspace recovery (RSR) problem and the single-column formulation~\eqref{eq:ODL sphere} of the orthogonal dictionary learning (DL) problem.


\subsection{Sharpness: Weak sharp minima}\label{subsec:sharp}

To begin, let us introduce the notion of a weak sharp minima set.

\begin{defi}[Sharpness; cf.~\cite{burke1993weak,li2011weak,karkhaneei2019nonconvex}]\label{def:sharpness}
	We say that $\calX\subseteq\stiefel(n,r)$ is a set of \emph{weak sharp minima} for the function $h:\R^{n\times r}\rightarrow\R$ with parameter $\alpha>0$ if there exists a constant $\rho>0$ such that for any $\mX\in \calB:=\{ \mX \in \R^{n\times r}: \dist(\mX,\calX) < \rho \}\cap\stiefel(n,r)$, we have
	\[
	h(\mX) - h(\mY) \geq \alpha \dist(\mX, \calX)
	\]
	for all $\mY\in \calX$, where  $\dist(\mX,\calX) := \inf_{\mY\in \calX} \|\mY - \mX\|_F$.
\end{defi}

From the definition, it is immediate that if $\calX$ is a set of weak sharp minima for $h$, then it is the set of minimizers of $h$ over $\calB$, and the function value grows linearly with the distance to $\calX$. Moreover, if $h$ is continuous (e.g., when $h$ is weakly convex), then $\calX$  can be taken as closed.

Similar notions of sharpness play a fundamental role in establishing the linear convergence of a host of methods for weakly convex minimization in the Euclidean space. For instance, it is shown in~\cite{goffin1977convergence} that the subgradient method with geometrically diminishing stepsizes will converge linearly to the optimal solution set when applied to minimize a sharp convex function. Later, the work~\cite{davis2018subgradient} establishes a similar linear convergence result for sharp weakly convex minimization. In the recent work~\cite{li2019incremental}, it is shown that the incremental subgradient, proximal point, and prox-linear methods will converge linearly when applied to minimize a sharp weakly convex function. In this paper, we extend, for the first time, the above results to the manifold setting by establishing the linear convergence of Riemannian subgradient-type methods for minimizing a weakly convex function over the Stiefel manifold under the sharpness property in \Cref{def:sharpness}.

\subsection{Riemannian subgradient and incremental subgradient methods}

Again, we will focus on analyzing the Riemannian incremental subgradient method. The analysis of the Riemannian subgradient method will follow as a special case. We first present the following result, which is crucial for our subsequent development.
\begin{prop}\label{prop:stiefel linear rate recursion}
Under the setting of \Cref{prop:rigd global rate}, for any $\overline \mX\in \stiefel(n,r)$, we have
	\[
	\begin{aligned}
	\left\| \mX_{k+1}  - \overline \mX \right\|_F^2 &\leq (1+2m\gamma_k (L+\tau))\left\|  \mX_{k} - \overline \mX \right\|_F^2 - 2m\gamma_k \left( f(\mX_{k}) - f(\overline \mX)  \right) \\
	&\quad + m^2 \gamma_k^2L^2 + C(m) \gamma_k^3 L^2(L+\tau), \ \ \forall \ k\geq 0,
	\end{aligned}
	\]
	where $C(m) = \tfrac{1}{3}m(m-1)(2m-1)$. 
\end{prop}
\begin{proof}
	According to \Cref{lem:nonexpansive polar retra}, for any $\overline \mX\in \stiefel(n,r)$, we have
	\e\label{eq:stiefel recursion 1}
	\begin{split}
		\left\| \mX_{k,i}  - \overline \mX \right\|_F^2 &\leq \left\|  \mX_{k,i-1} + \vxi_{k,i-1}  - \overline \mX \right\|_F^2 \\
		&\stackrel{(i)}{\leq} \left\|  \mX_{k,i-1} - \overline \mX \right\|_F^2 - 2 \gamma_k \left\langle \widetilde \nabla_{\calR} f_i(\mX_{k,i-1}),  \mX_{k,i-1} - \overline \mX    \right\rangle + \gamma_k^2 L^2 \\
		&\stackrel{(ii)}{\leq } \left\|  \mX_{k,i-1} - \overline \mX \right\|_F^2 - 2\gamma_k \left( f_i(\mX_{k,i-1}) - f_i(\overline \mX)  \right)\\
		&\quad + \gamma_k (L+\tau) \left\|  \mX_{k,i-1} - \overline \mX \right\|_F^2 + \gamma_k^2L^2,
	\end{split}
	\ee
	where $(i)$ follows from the fact that $\left\| \widetilde \nabla_{\calR} f_i(\mX_{k,i-1}) \right\|   \leq \left\| \widetilde \nabla f_i(\mX_{k,i-1})  \right\|_F  \leq L$ and $(ii)$ is from \Cref{thm:Riemannian subgradient inequality}.	Following the derivations of \eqref{eq:inner iterates distance sitefel}--\eqref{eq:iterates bound}, we get
	\begin{align*}
	f_i(\overline \mX) - f_i(\mX_{k,i-1}) &\leq (i-1) \gamma_k L^2 - (f_i(\mX_k) -  f_i(\overline \mX) ), \\
	\left\|  \mX_{k,i-1} - \overline \mX \right\|_F^2 &\leq 2 (i-1)^2 \gamma_k^2 L^2 + 2  \left\|\mX_k - \overline \mX \right\|_F^2.
	\end{align*}
	Substituting the above two upper bounds into \eqref{eq:stiefel recursion 1} gives
	\begin{align*}
		\left\| \mX_{k,i}  - \overline \mX \right\|_F^2 &\leq \left\|  \mX_{k,i-1} - \overline \mX \right\|_F^2 - 2\gamma_k \left( f_i(\mX_{k}) - f_i(\overline \mX)  \right) + 2\gamma_k (L+\tau) \left\|  \mX_{k} - \overline \mX \right\|_F^2 \\
		&\quad + (2i-1) \gamma_k^2L^2 + 2 (i-1)^2 \gamma_k^3 L^2(L+\tau).
	\end{align*}
	Upon summing both sides of the above inequality over $i=1,\ldots,m$, we obtain
	\begin{align*}
		\left\| \mX_{k+1}  - \overline \mX \right\|_F^2 &\leq (1+2m\gamma_k (L+\tau))\left\|  \mX_{k} - \overline \mX \right\|_F^2 - 2m\gamma_k \left( f(\mX_{k}) - f(\overline \mX)  \right)  \\
		&\quad + m^2 \gamma_k^2L^2 + \frac{1}{3}m(m-1)(2m-1)\gamma_k^3 L^2(L+\tau),
	\end{align*}
	which completes the proof.
\end{proof}

In order for Riemannian subgradient-type methods to achieve linear convergence when solving sharp instances of problem~\eqref{eq:stiefel opt problem}, we need to choose the stepsizes appropriately. Motivated by previous works~\cite{goffin1977convergence,Shor:1985:MMN:3585,nedic2001convergence,davis2018subgradient,li2019incremental} on sharp weakly convex minimization in the Euclidean space, let us consider using geometrically diminishing stepsizes of the form $\gamma_k = \beta^k \gamma_0$, $k=0,1,\ldots$. Then, by applying \Cref{prop:stiefel linear rate recursion}, we can establish the following local linear convergence result:

\begin{thm}\label{thm:stiefel linear rate}
Consider the setting of \Cref{prop:rigd global rate}. Suppose further that $\calX$ is a set of weak sharp minima for the objective function $f$ in~\eqref{eq:stiefel opt problem} with parameter $\alpha>0$ over the set $\calB$ defined in \Cref{def:sharpness}.	Let $\{\mX_k\}$ be the sequence generated by Riemannian incremental subgradient method \eqref{eq:rigd update} for solving problem \eqref{eq:stiefel opt problem}, in which the initial point $\mX_0$ satisfies $ \dist(\mX_0, \calX) < \min\left\{ \tfrac{\alpha}{L+\tau}, \rho \right\} $ (so that $\mX_0 \in \calB$) and the stepsizes satisfy $\gamma_k = \beta^k \gamma_0$, $k=0,1,\ldots$, where
\[ \gamma_0< \min \left\{ \frac{2me_0(\alpha - (L+\tau)e_0)}{d(m)L^2}, \frac{e_0}{2m(\alpha - (L+\tau) e_0)} \right\}, \]
\[ \beta \in [\beta_{\min},1) \quad \text{with} \quad \beta_{\min}:= \sqrt{  1+ 2m\left(L+\tau - \frac{\alpha}{e_0} \right) \gamma_0 + \frac{d(m)L^2}{e_0^2} \gamma_0^2 }, \]
\[ d(m) = \frac{5}{3} m^2 - m + \frac{1}{3}, \quad\text{and}\quad e_0 = \min\left\{  \max\left\{ \dist(\mX_0,\calX), \frac{\alpha}{2(L+\tau)} \right\}, \rho \right\}. \]
Then, we have
	\[
	\dist(\mX_k, \calX)\leq \beta^k \cdot e_0, \ \ \forall \ k\geq 0.
	\]
\end{thm}
\begin{proof}
	We first show that $\beta_{\min} \in (0,1)$ and $\gamma_0 > 0$ are well defined. Towards that end, note that $\beta_{\min} = \sqrt{  1+ v(\gamma_0) }$  with $v(\gamma) = 2m\left(L+\tau - \tfrac{\alpha}{e_0} \right) \gamma + \frac{d(m)L^2}{e_0^2} \gamma^2$ being quadratic in $\gamma$. By definition of $\gamma_0$, we immediately have $v(\gamma_0)<0$. Moreover, the function $\gamma\mapsto v(\gamma)$ attains its minimum at $\overline \gamma = \frac{me_0(\alpha -(L+\tau)e_0)}{d(m) L^2}$ with value $v(\overline \gamma) =  - \frac{m^2(\alpha - (L+\tau)e_0)^2}{d(m) L^2} > - \frac{\alpha^2}{L^2} \geq -1$, where the first inequality is due to $\tfrac{m^2}{d(m)}\le1$ for $m\ge1$ and $e_0<\tfrac{\alpha}{L+\tau}$, and the second inequality is implied by the sharpness assumption because $ \alpha \|\mX - \overline \mX\|_F \leq f(\mX) - f(\overline \mX) \leq L\|\mX - \overline \mX\|_F $ for any $\mX\in\calB$ and  $\overline \mX \in \calP_{\calX} (\mX)$. Hence, we have $v(\gamma_0) \in (-1,0)$, which implies that $\beta_{\min}\in(0,1)$.  On the other hand, since $e_0< \tfrac{\alpha}{L+\tau}$, the upper bound on the initial stepsize $\gamma_0$ is positive. It follows that $\gamma_0$ is well defined.
		
	We now prove the theorem by induction on $k$. The base case $k = 0$ follows directly from the definition of $e_0$. For the inductive step, suppose that $\dist(\mX_k, \calX)\leq \beta^k \cdot e_0$ for some $k\ge0$. Note that this implies $\mX_k \in \calB$. Let $\overline\mX \in \calP_{\calX}(\mX_k)$. Clearly, we have $\dist( \mX_{k} , \calX) = \left\|  \mX_{k} - \overline \mX \right\|_F$ and $\dist( \mX_{k+1} , \calX) \leq \left\|  \mX_{k+1} - \overline \mX \right\|_F$. Hence, by \Cref{prop:stiefel linear rate recursion}, the sharpness assumption, and the fact that $\gamma_k \leq  \gamma_0$ for $k=0,1,\ldots$, we get
	\e\label{eq:key recursion}
	\begin{split}
		\dist^2( \mX_{k+1} , \calX) &\leq (1+2m\gamma_0 (L+\tau))\dist^2( \mX_{k} , \calX) - 2m  \gamma_k \alpha \dist( \mX_{k} , \calX)  \\
		&\quad + m^2 \gamma_k^2L^2 + C(m) \gamma_k^3 L^2(L+\tau).
	\end{split}
	\ee
	Observe that the RHS of the above recursion is quadratic in $\dist(\mX_{k},\calX)$. By definition of $\gamma_0$, we have $\gamma_0< \frac{e_0}{2m(\alpha - (L+\tau) e_0)}$ and hence $\tfrac{2m\gamma_0\alpha}{1+2m\gamma_0(L+\tau)} < e_0$. This implies that the RHS of \eqref{eq:key recursion} achieves its maximum when $\dist(\mX_{k},\calX) = \beta^k \cdot e_0$. Since $\dist(\mX_{k},\calX) \leq \beta^k \cdot e_0$ by the inductive hypothesis, plugging $\gamma_k = \beta^k \gamma_0$ and $\dist(\mX_{k},\calX) = \beta^k \cdot e_0$ into \eqref{eq:key recursion} yields
	\e \label{eq:recursion-9}
	\begin{split}
		&\dist^2(\mX_{k+1},\calX) \\
		\leq& \ \beta^{2k} e_0^2 \left[ 1+ 2m\left(L+\tau  - \frac{\alpha}{e_0} \right)\gamma_0  + L^2 \left( \frac{ m^2 + C(m) \gamma_0 (L+\tau) }{e_0^2} \right) \gamma_0^2  \right]. 
	\end{split}
	\ee
	Note that $\gamma_0< \frac{2m(\alpha e_0 - (L+\tau)e_0^2)}{d(m)L^2} \leq \frac{m\alpha^2}{2d(m)L^2(L+\tau)} < \frac{1}{m(L+\tau)}$. It then follows from \eqref{eq:recursion-9} that
	\[ 
		\dist^2(\mX_{k+1},\calX) \leq \beta^{2k} e_0^2 \left[ 1+ 2m\left(L+\tau  - \frac{\alpha}{e_0} \right)\gamma_0  +   \frac{ d(m) L^2 }{e_0^2} \gamma_0^2  \right]  \leq \beta^{2(k+1)} e_0^2.
	\] 
	This completes the inductive step and hence the proof of \Cref{thm:stiefel linear rate}. 
\end{proof}

From \Cref{thm:stiefel linear rate}, we see that in order to achieve a fast linear convergence rate, one should choose an appropriate $\gamma_0$ so that the minimum decay factor $\beta_{\min}$ is as small as possible.  By minimizing $\beta_{\min}$ with respect to $\gamma_0$, we see that the theoretical minimum value of $\beta_{\min}$ is $\sqrt{1- \tfrac{m^2(\alpha - (L+\tau)e_0)^2}{d(m) L^2}}$, which is attained at $\gamma_0 = \overline \gamma_0 = \tfrac{me_0(\alpha -(L+\tau)e_0)}{d(m) L^2}$. This suggests that subject to the requirement in \Cref{thm:stiefel linear rate}, the initial stepsize $\gamma_0$ should be set as close to $\overline \gamma_0$ as possible. As an illustration, consider the case where the sharpness property holds globally over the Stiefel manifold (i.e., $\calB = \stiefel(n,r)$ in \Cref{def:sharpness}). Then, the parameter $\rho$ can be set as large as possible. In this case, we have $e_0 =  \max\left\{ \dist(\mX_0,\calX), \tfrac{\alpha}{2(L+\tau)} \right\}$, and the condition on $\gamma_0$ in \Cref{thm:stiefel linear rate} becomes $\gamma_0<   \tfrac{2me_0(\alpha - (L+\tau)e_0)}{d(m)L^2}$. This implies that we can choose $\gamma_0 = \overline \gamma_0$ to obtain the smallest possible $\beta_{\min}$. Note, however, that the larger the initialization error $\dist(\mX_0,\calX)$, the larger  the minimum decay factor $\beta_{\min}$. In particular, from the expression for $\beta_{\min}$ above, we see that $\beta_{\min}$ approaches $1$ as $\dist(\mX_0,\calX)$ approaches its maximum $\frac{\alpha}{L+\tau}$.


We end this section by comparing the sharpness property with the  \emph{Riemannian regularity condition} used in \cite{bai2018subgradient} and \cite{zhu2019grasssub} for orthogonal DL and RSR, respectively. For a target solution set $\calX$, the Riemannian regularity condition stipulates the existence of a constant $\kappa>0$ such that $\left \langle    \widetilde \nabla_{\calR} f(\mX), \mX - \mY \right\rangle \geq \kappa  \dist(\mX,\calX)$ for all $\mX$ in a \emph{small neighborhood} of $\calX$ and $\mY \in \calP_{\calX}(\mX)$. This condition is motivated by the need to bound the inner product term on the LHS in the convergence analysis of the Riemannian subgradient method; see \eqref{eq:stiefel recursion 1} with $f_i = f$ and $\mX_{k,i-1} = \mX_k$. Informally, the Riemannian regularity condition is a combination of the Riemannian subgradient inequality in \Cref{thm:Riemannian subgradient inequality} and the sharpness property in \Cref{def:sharpness}. However, the tangling of these two elements potentially restricts the applicability of the Riemannian regularity condition. In particular, since the Riemannian regularity condition can only hold locally, it cannot be used to establish global convergence and iteration complexity results for the Riemannian subgradient method.

\section{Extension to Optimization over a Compact Embedded Submanifold} \label{sec:extension}

There is of course no conceptual difficulty in adapting the Riemannian subgradient-type methods in \Cref{subsec:methods} to minimize weakly convex functions over more general manifolds. All that is needed is an efficiently computable retraction on the manifold of interest. In this section, let us briefly demonstrate how the machinery developed in the previous sections can be extended to study the convergence behavior of Riemannian subgradient-type methods when the manifold in question is compact and defined by a certain smooth mapping.


\paragraph{Riemannian subgradient inequality} Our starting point is the following generalization of the Riemannian subgradient inequality in \Cref{thm:Riemannian subgradient inequality}, which applies to restrictions of weakly convex functions on a class of compact embedded submanifolds of the Euclidean space. Some examples of manifolds in this class include the generalized Stiefel manifold, oblique manifold, and symplectic manifold; see, e.g.,~\cite{absil2009optimization}.
\begin{cor}\label{cor:Riemannian subgradient inequality}
Let $\calM$ be a compact submanifold of $\R^p$ given by $\calM=\{\mX \in \R^p : F(\mX )= \mathbf{0}\}$, where $F: \mathbb{R}^p \rightarrow \mathbb{R}^q$ is a smooth mapping whose derivative $D F(\mX)$ at $\mX$ has full row rank for all $\mX \in \calM$. Then, for any weakly convex function $h:\R^p\rightarrow\R$, there exists a constant $c>0$ such that
	\[ 	h(\mY) \geq h(\mX) + \left\langle  \widetilde \nabla_{\calR} h(\mX)  , \mY-\mX   \right\rangle  - c \|\mY - \mX\|_F^2 \]
	for all $\mX,\mY \in \calM$ and $\widetilde \nabla_{\calR} h(\mX) \in \partial_{\calR} h(\mX)$.
\end{cor}
\begin{proof}
	By our assumptions on $F$ and \cite[Equation (3.19)]{absil2009optimization}, we have $\T_{\mX}\calM= \text{ker}(DF(\mX))$, where $\text{ker}(T)$ denotes the kernel of the operator $T$.  Thus, the projector $\calP_{\T_{\mX}\calM}^{\perp}$ is given by $DF(\mX)^\top (DF(\mX)DF(\mX)^\top)^{-1}DF(\mX)$.  Following the proof of \Cref{thm:Riemannian subgradient inequality}, we need to bound 
	\begin{align*}
	&\left\langle \widetilde \nabla h(\mX),\calP_{\T_{\mX}\calM}^{\perp} (\mY - \mX) \right\rangle \ge -\| \widetilde \nabla h(\mX)\|_F \cdot \| \calP_{\T_{\mX}\calM}^{\perp} ( \mY - \mX) \|_F \\
	=& -\| \widetilde \nabla h(\mX)\|_F \cdot \| DF(\mX)^\top (DF(\mX)DF(\mX)^\top)^{-1}DF(\mX)( \mY - \mX) \|_F \\
	\geq& -\| \widetilde \nabla h(\mX)\|_F \cdot \max_{\mX \in \calM} \| DF(\mX)^\top (DF(\mX)DF(\mX)^\top)^{-1}\|_F \cdot \| DF(\mX)( \mY - \mX) \|_F.
	\end{align*}
	Since $h$ is weakly convex on $\R^p$, it is Lipschitz continuous on { any bounded open convex set $\calU$ that contains $\calM$}. Thus, the term $\| \widetilde \nabla h(\mX)\|_F$ is bounded above. Moreover, the compactness of $\calM$ implies that the term $\max_{\mX \in \calM} \| DF(\mX)^\top (DF(\mX)DF(\mX)^\top)^{-1}\|_F$ is also bounded above. Lastly, observe that $F(\mY ) = F(\mX) + DF(\mX)( \mY - \mX) + \mathcal{O} ( \| \mY -\mX \|_F^2)$ by Taylor's theorem and $F(\mX)=F(\mY)=\bm{0}$ whenever $\mX,\mY\in\calM$. Hence, we have $\| DF(\mX)( \mY - \mX) \|_F = \mathcal{O} ( \| \mY -\mX \|_F^2)$. Putting these together, we conclude that $\left| \left\langle \widetilde \nabla h(\mX),\calP_{\T_{\mX}\calM}^{\perp} ( \mY - \mX) \right\rangle \right| = \mathcal{O} ( \| \mY -\mX \|_F^2)$. The rest of the argument is similar to that in the proof of \Cref{thm:Riemannian subgradient inequality}. 
\end{proof}

\paragraph{General retractions} The notion of retraction introduced in \Cref{sec:opt over stiefel} for the Stiefel manifold can be easily adapted to that for general manifolds. Specifically, a retraction on the manifold $\calM$ is a smooth map $\Retr:\T\calM \rightarrow \calM$ from the tangent bundle $\T\calM$ onto the manifold $\calM$ that satisfies $\Retr_{\mX}(\bm{0})=\mX$ and $D\Retr_{\mX}(\bm{0})=\bm{Id}$ for all $\mX\in\calM$. Unlike the polar decomposition-based retraction on the Stiefel manifold, a general retraction may not have the Lipschitz-like property in~\Cref{lem:nonexpansive polar retra}. Nevertheless, a retraction on a compact submanifold $\calM$ satisfies a \emph{second-order boundedness} property~\cite{boumal2018global}; i.e., there exists a constant $b\ge0$ such that for all $\mX \in \calM$ and $\vxi \in \T_{\mX}\calM$,
\[ \| \Retr_{\mX}(\vxi) - \mX - \vxi \|_F \le b \|\vxi\|_F^2. \]
This allows us to replace the result in~\Cref{lem:nonexpansive polar retra} by
\begin{align*}
\left\|\Retr_{\mX}(\vxi) - \overline \mX \right\|_F 	&= \| (\mX + \vxi)  - \overline \mX + \Retr_{\mX}(\vxi) - (\mX+\vxi) \|_F \\
&\le \|\mX + \vxi  - \overline \mX\|_F + b \|\vxi\|_F^2,
\end{align*}
which holds for any $\mX,\overline \mX\in \calM$ and $\vxi\in\T_{\mX}\calM$. Although the above inequality has the extra term $b\|\vxi\|_F^2$, it can still be used to establish convergence guarantees (with slightly worse constants) for the Riemannian subgradient-type methods considered in \Cref{subsec:methods}. Specifically, by following the analyses in~\Cref{sec:global convergence,sec:linear convergence}, we can show that for problem~\eqref{eq:stiefel opt problem} with the Stiefel manifold $\stiefel(n,r)$ being replaced by a manifold of the type considered in \Cref{cor:Riemannian subgradient inequality}, the iteration complexity of Riemannian subgradient-type methods for computing an $\varepsilon$-nearly stationary point is $\calO(\varepsilon^{-4})$, and the Riemannian subgradient and incremental subgradient methods will achieve a local linear convergence rate if the instance satisfies the sharpness property in \Cref{def:sharpness}.

\section{Applications and Numerical Results}\label{sec:application}

In this section, we apply the Riemannian subgradient-type methods in \Cref{subsec:methods} to solve the RSR and orthogonal DL problems. As described in \Cref{sec:introduction}, the objective functions of both problems are weakly convex. Thus, \Cref{thm:rigd global rate} and \Cref{thm:stochastic subgradient global rate} ensure that the Riemannian subgradient-type methods with arbitrary initialization will have a global convergence rate of $\calO({k^{-1/4}})$ when utilized to solve those problems. We also discuss the sharpness properties of the RSR and orthogonal DL problems.
For reproducible research, our code for generating the numerical results can be found at
\begin{center}
	\url{https://github.com/lixiao0982/Riemannian-subgradient-methods}
\end{center}

\subsection{Robust subspace recovery (RSR)} \label{subsec:RSR}
We begin with the DPCP formulation \eqref{eq:dpcp stiefel} of the RSR problem, which has a relatively simpler form than the least absolute deviation (LAD) formulation \eqref{eq:RSR}. Recall that the objective function in~\eqref{eq:dpcp stiefel} takes the form $ \stiefel(n,r) \ni \mX \mapsto f(\mX) = \frac{1}{m} \sum_{i=1}^m \left\|\widetilde \vy_i^\top \mX \right\|_2 $, where $\widetilde \vy_i\in\R^n$ ($i=1,\ldots,m$) denotes the $i$-th column of $\widetilde \mY = \begin{bmatrix} \mY & \mO \end{bmatrix} \mGamma \in \R^{n\times m}$, the columns $\vy_i$ of $\mY \in \R^{n\times m_1}$ form inlier points spanning a $d$-dimensional subspace $\calS$ with $r=n-d$, the columns $\vo_i$ of $\mO\in\R^{n\times m_2}$ form outlier points, and $\mGamma \in \R^{m\times m} $ is an unknown permutation. Note that $f$ is rotationally invariant; i.e., $f(\mX) = f(\mX\mR)$ for any $\mX\in\stiefel(n,r)$ and $\mR\in \stiefel(r,r)$. 

\paragraph{Sharpness} Let $\mS^\perp \in \stiefel(n,r)$ be an orthonormal basis of $\calS^{\perp}$.
Since the goal of DPCP is to find an orthonormal basis (but not necessary $\mS^\perp$) for $\calS^\perp$, we are interested in the elements in the set $\calX = \{\mS^\perp\mR \in \R^{n\times r}:\mR\in \stiefel(r,r)\}$. Due to rotation invariance, $f$ is constant on $\calX$. To study the sharpness property of problem \eqref{eq:dpcp stiefel}, let us introduce two quantities that reflect how well distributed the inliers and outliers are:
\begin{align}
 & c_{\mY,\min}:=\frac{1}{m_1}\inf_{\mD\in\R^{n\times \ell}, \atop { \|\mD\|_F =1,\, {\rm col}(\mD) \subseteq \calS }} \ \sum_{i=1}^{m_1} \|\vy_i^\top\mD\|_2, 
\label{eq:cXmin} \\
 &  c_{\mO,\max}:=\frac{1}{m_2}\sup_{\mB\in \R^{n\times r}, \atop \|\mB\|_F=1} \ \sum_{i=1}^{m_2} \|\vo_i^\top\mB\|_2.  
\label{eq:cOmax}
\end{align}
 Here, $\ell = \min\{d,r\}$ and ${\rm col}(\mD)$ denotes the column space of $\mD$.
 In a nutshell, larger values of $c_{\mY,\min}$ (respectively, smaller values of $c_{\mO,\max}$) correspond to a more uniform distribution of inliers (respectively, outliers). As the following proposition shows, the quantities $c_{\mY,\min}$ and $c_{\mO,\max}$ can be used to capture the sharpness property of the DPCP formulation \eqref{eq:dpcp stiefel}.

\begin{prop}\label{prop:sharp of DPCP}
	Suppose that $m_2c_{\mO,\max} \le \frac{1}{2}  m_1c_{\mY,\min}$. Then, the DPCP formulation~\eqref{eq:dpcp stiefel} has $\calX$ as a set of weak sharp minima with parameter $ \alpha = \frac{1}{m} \big( \frac{1}{2}  m_1c_{\mY,\min}  - m_2 c_{\mO,\max}  \big)  >0 $ over the set $\calB=\stiefel(n,r)$; i.e., 
	\[
	f(\mX) - f(\mS^\perp) \ge \alpha \dist(\mX,\calX),  \ \ \forall \ \mX\in \stiefel(n,r).
	\]
\end{prop}
\begin{proof}
	Let $ \mX \in\stiefel(n,r)$ be arbitrary. For any $\mX^\star \in \calP_{\calX}(\mX)$, we have $f(\mX^\star)=f(\mS^\perp)$ and 
\e\label{eq:dpcp sharp 1}
\begin{split}
	f(\mX) - f(\mS^\perp) &= \frac{1}{m}  \sum_{i=1}^m \left\|\widetilde \vy_i^\top \mX \right\|_2 -  \frac{1}{m}  \sum_{i=1}^m \left\|\widetilde \vy_i^\top \mS^\perp  \right\|_2\\
	& = \frac{1}{m}  \sum_{i=1}^{m_1} \left\| \vy_i^\top \mX \right\|_2  +  \frac{1}{m} \left(  \sum_{i=1}^{m_2} \left\| \vo_i^\top \mX \right\|_2 -  \sum_{i=1}^{m_2} \left\| \vo_i^\top \mS^\perp  \right\|_2 \right),
\end{split}
\ee
where $\vy_i$ (respectively, $\vo_i$) is the $i$-th column of $\mY$ (respectively, $\mO$), and the second line follows because the inliers $\{\vy_i\}_{i=1}^{m_1}$ are orthogonal to $\mS^\perp$. Now, let us derive lower bounds for the two terms on the right-hand side separately. 

For the first term, let $\mS\in\R^{n\times d}$ be an orthonormal basis of the subspace $\calS$. By projecting $\mX$ onto the orthogonal subspaces $\calS$ and $\calS^\perp$, we have
	\e \label{eq:decompose X}
	   \mX = \mS\mS^\top\mX + \mS^\perp(\mS^\perp)^\top\mX.
	\ee
	For $i = 1,\ldots, r$, let $\phi_i = \arccos(\sigma_i((\mS^\perp)^\top \mX))$ be the $i$-th smallest principal angle between the subspaces spanned by $\mX$ and $\calS^\perp$, where $\sigma_i(\cdot)$ denotes the $i$-th largest singular value \cite{stewart1990matrix}. Then, we can write $(\mS^\perp)^\top\mX = \mU\cos(\mPhi)\mW^\top$, where $\cos(\mPhi) \in \R^{r\times r}$ is the diagonal matrix with $\cos(\phi_1) \ge \cdots \ge \cos(\phi_r)$ on its diagonal and $\mU\in\R^{r\times r}$, $\mW\in\R^{r\times r}$ are orthogonal matrices. 
	On the other hand, according to \cite[Theorem 2.7]{knyazev2007majorization}, the $i$-th smallest principal angle between the subspaces spanned by $\mX$ and $\calS$ is $\widetilde{\phi}_i = \frac{\pi}{2}-\phi_{r-i+1}$, where $i=1,\ldots,\ell$ with $\ell =\min\{d,r\}$. Hence, we can write $\mS^\top\mX = \mV\sin(\widetilde\mPhi)\mH^\top$, where $\sin(\widetilde\mPhi)\in \R^{\ell\times \ell}$ is the diagonal matrix with $\sin(\phi_r) \ge \cdots \ge \sin(\phi_{r-\ell+1})$ on its diagonal and $\mV\in\R^{d\times \ell}$, $\mH \in \R^{r\times \ell}$ are orthogonal matrices. These, together with \eqref{eq:decompose X}, yield $\mX = \mS \mV\sin(\widetilde\mPhi) \mH^\top + \mS^\perp\mU\cos(\mPhi) \mW^\top$. Hence, we can bound
	\e \label{eq:dpcp term 1}
	\begin{split}
		&\sum_{i=1}^{m_1}\|\vy_i^\top \mX\|_2 =  \sum_{i=1}^{m_1}\|\vy_i^\top \mS \mV\sin(\widetilde\mPhi) \|_2 \\
		=& \ \|\mS \mV\sin(\widetilde\mPhi)\|_F  \sum_{i=1}^{m_1}\left\|\vy_i^\top \frac{\mS \mV\sin(\widetilde\mPhi)}{\|\mS \mV\sin(\widetilde\mPhi)\|_F} \right\|_2 
		 \ge m_1 c_{\mY,\min} \|\sin(\widetilde\mPhi)\|_F,
	\end{split}
	\ee
	where  $c_{\mY,\min}$ is defined in \eqref{eq:cXmin}. On the other hand, observe that
	\e\label{eq:dist sin}
	\begin{split}
		&\dist^2(\mX,\calX) = \minimize_{\mR\in \stiefel(r,r)}\|\mX - \mS^\perp \mR\|_F^2 =  \|\mX - \mS^\perp \mU\mW^\top\|_F^2 \\
		=& \ 2 r - 2\trace(\cos(\mPhi)) = 2\sum_{i=1}^{\ell} (1 -\cos(\phi_i))
		= 4\sum_{i=1}^{\ell}\sin^2(\phi_{i}/2) \leq 4 \|\sin(\widetilde\mPhi)\|_F^2,
		\end{split}
	\ee
	where the second equality follows from the solution to the orthogonal Procrustes problem~\cite{S66} and the fourth equality utilizes the fact that the number of nonzero principal angles in $\mPhi$ is at most $\ell = \min\{d,r\}$ \cite[Theorem 2.7]{knyazev2007majorization}. Combining \eqref{eq:dpcp term 1} and \eqref{eq:dist sin} gives 
    \e \label{eq:dpcp term 1 bound}
    	\sum_{i=1}^{m_1}\|\vy_i^\top \mX\|_2 \geq  \frac{1}{2}{m_1 c_{\mY,\min}} \dist(\mX,\calX).
    \ee
	
	Now, let us consider the second term on the right-hand side of \eqref{eq:dpcp sharp 1}. Let $\mR^\star = \argmin_{\mR\in \stiefel(r,r)} \|\mX - \mS^{\perp}\mR\|_F$. Then, we have
	\begin{equation}\begin{split}
	&\left|\sum_{i=1}^{m_2}\|\vo_i^\top \mX\|_2 - \|\vo_i^\top \mS^\perp\|_2 \right| \le \sum_{i=1}^{m_2} \left|\|\vo_i^\top \mX\|_2 - \|\vo_i^\top \mS^\perp\mR^\star\|_2 \right|  \\
	\le& \ \sum_{i=1}^{m_2}\|\vo_i^\top (\mX- \mS^\perp\mR^\star)\|_2 = \|\mX- \mS^\perp\mR^\star\|_F \sum_{i=1}^{m_2}\left\|\vo_i^\top \frac{\mX- \mS^\perp\mR^\star}{\|\mX- \mS^\perp\mR^\star\|_F}\right\|_2\\
	 \le& \  m_2  c_{\mO,\max} \dist(\mX,\calX), 
	\end{split}\label{eq:dpcp term 2}\end{equation}
	where   $c_{\mO,\max}$ is defined in \eqref{eq:cOmax}.
	
	By plugging \eqref{eq:dpcp term 1 bound} and \eqref{eq:dpcp term 2} into \eqref{eq:dpcp sharp 1}, the desired result follows.
\end{proof}

The requirement $m_2c_{\mO,\max} \le \frac{1}{2}  m_1c_{\mY,\min}$ in \Cref{prop:sharp of DPCP} determines the number of outliers that can be tolerated.  Now, let us give probabilistic estimates of the quantities $c_{\mY,\min}$ and $c_{\mO,\max}$ under the popular \emph{Haystack model} (see, e.g.,~\cite{lerman2015robust,maunu2019well,zhang2014novel}) of the input data. The model stipulates that the inliers $\{\vy_i\}_{i=1}^{m_1}$ are i.i.d.~according to the Gaussian distribution $\calN(\mathbf{0},\frac{1}{d}\calP_{\calS})$ with $\calP_{\calS}$ being the orthogonal projector onto the $d$-dimensional subspace $\calS$, while the outliers $\{\vo_i\}_{i=1}^{m_2}$ are i.i.d.~according to the Gaussian distribution $\calN(\mathbf{0},\frac{1}{n}\mId_n)$.

\begin{lem}\label{lem:statistics of outliers and inliers}
	Under the Haystack model, the event
	\[
	     c_{\mY,\min} \geq \sqrt{\frac{2}{d\pi}} - \sqrt{\frac{8\ell}{m_1}} - \frac{c_1}{\sqrt{m_1}}
	\]
	will hold with probability at least $1-2\exp(-\frac{c_1^2d}{2})$ for some constant $c_1>0$, where $\ell = \min\{d,r\}$. Moreover, the event
	\[
	c_{\mO,\max}   \leq  \frac{1}{\sqrt{n}} + \sqrt{\frac{8r}{m_2}} + 	\frac{c_2}{\sqrt{m_2}}
	\]
	will hold with probability at least $1-2\exp(-\frac{c_2^2n}{2})$ for some constant $c_2>0$.
\end{lem}

The proof of \Cref{lem:statistics of outliers and inliers} can be found in \Cref{sec:appendix}. \Cref{lem:statistics of outliers and inliers} implies that under the Haystack model, if the numbers of inliers $m_1$ and outliers $m_2$ satisfy $m_1 \gtrsim d\ell$ and $m_2 \gtrsim nr$, then we will have $c_{\mY,\min} \gtrsim \frac{1}{\sqrt{d}}$ and $c_{\mO,\max} \lesssim \frac{1}{\sqrt{n}}$ with high probability. Combining \Cref{thm:stiefel linear rate}, \Cref{prop:sharp of DPCP}, and  \Cref{lem:statistics of outliers and inliers}, we see that as long as 
\e\label{eq:number of outliers}
  m_2 \ \lesssim \ \sqrt{\frac{n}{d}} m_1
\ee
so that $m_2c_{\mO,\max} \le \frac{1}{2}  m_1c_{\mY,\min}$, the Riemannian subgradient and incremental subgradient methods with geometrically diminishing stepsizes and a proper initialization will converge linearly to an orthonormal basis of $\calS^{\perp}$. One initialization strategy is to take the bottom eigenvectors of $\widetilde \mY \widetilde \mY^\top$ \cite{maunu2019well,zhu2018dual}.

It is instructive to compare the bound~\eqref{eq:number of outliers} with those in the literature. When $d=\calO(1)$ or when both $d$ and $r$ are on the order of $n$, our bound~\eqref{eq:number of outliers} holds in the regime $m \gtrsim n^2$. In this regime, the algorithms proposed in~\cite{zhang2014novel,maunu2019well} can recover $\calS$ as long as $m_2 \lesssim \frac{\sqrt{n(n-d)}}{d} m_1$; see~\cite[Section 5.5.2]{maunu2019well}. Such a bound is superior to ours when $d=\calO(1)$ but is comparable when both $d$ and $r$ are on the order of $n$. When $r=\calO(1)$, our bound~\eqref{eq:number of outliers} holds in the regime $m \gtrsim n$, which is superior to the bound $m_2 \lesssim \frac{n-d}{d}m_1$ established in~\cite{lerman2015robust,zhang2016robust} for the same regime. We remark that there are other works \cite{zhu2018dual,lerman2018fast,xu2012outlier} studying the RSR problem. However, they differ from our work in that they either assume different data models, require additional data structures, or consider the asymptotic setting $m \rightarrow \infty$.

To further demonstrate the power of \Cref{prop:sharp of DPCP}, let us use it to establish the sharpness property of the LAD formulation \eqref{eq:RSR}. To begin, let $g$ be the objective function in \eqref{eq:RSR} and $\mS\in\stiefel(n,d)$ be an orthonormal basis of $\calS$. We are interested in the set $\calD = \{ \mS\mR: \mR\in \stiefel(d,d) \}$, whose elements are different orthonormal bases of $\calS$. Now, observe that for any $\mX\in\stiefel(n,r)$, we can find an orthonormal basis $\mZ\in\stiefel(n,d)$ of $\text{col}(\mX)^\perp$, and vice versa, such that $f(\mX) = g(\mZ)$ (recall that $f$ is the objective function in~\eqref{eq:dpcp stiefel}). Hence, \Cref{prop:sharp of DPCP} asserts that $g(\mZ) - g(\mS) \geq  \alpha \dist(\mX, \calX)$. By invoking \cite[Theorem 2.7]{knyazev2007majorization}, we obtain $\dist(\mX, \calX) = \dist(\mZ, \calD)$, which shows that $\calD$ is a set of weak sharp minima with parameter $\alpha$ over the set $\calB = \stiefel(n,d)$.

\paragraph{Experiments}
\begin{figure}[t]
	\begin{subfigure}{0.48\linewidth}
		\centerline{
			\includegraphics[width=2.5in]{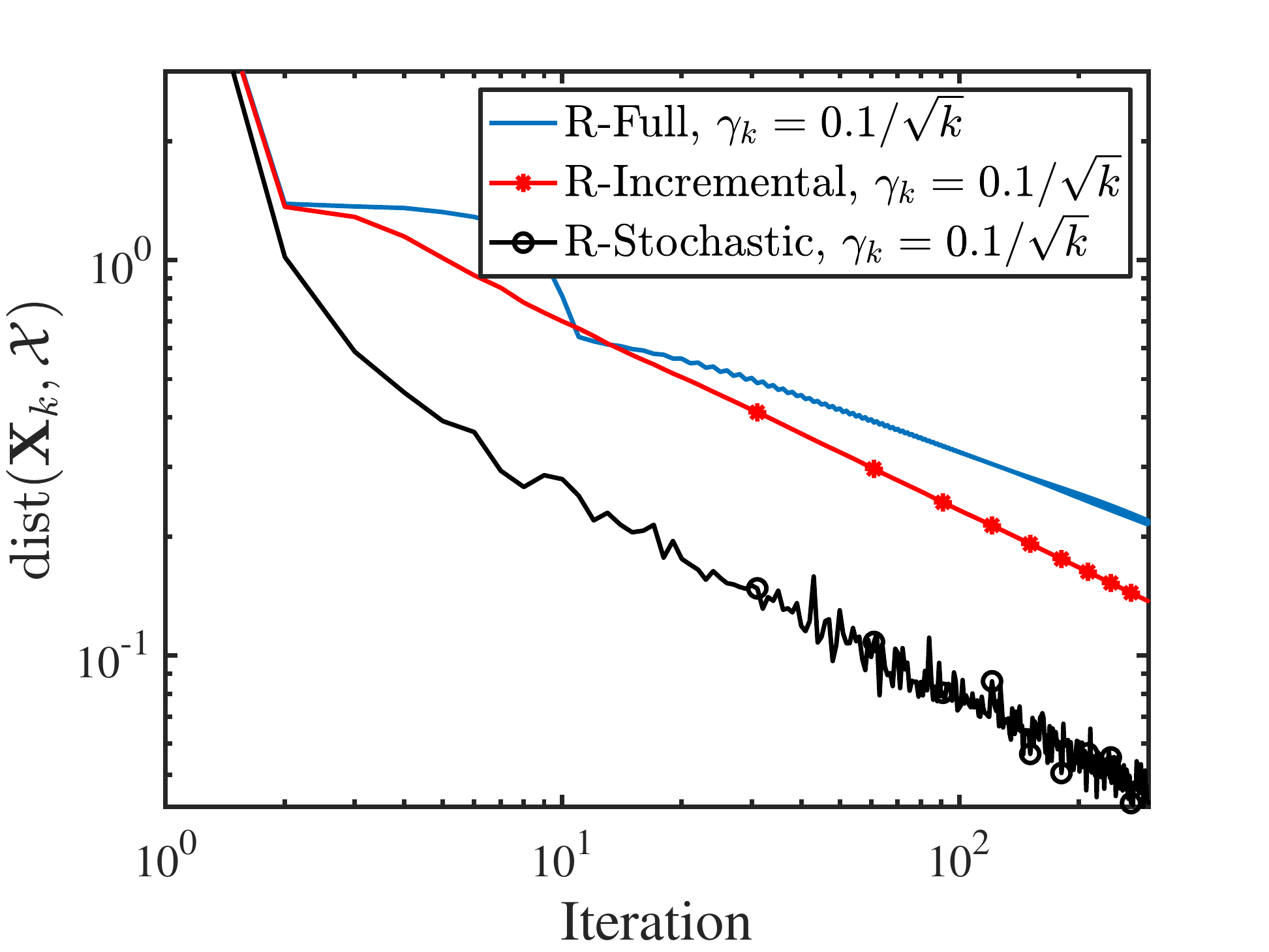}}
		\caption{\footnotesize DPCP: Diminishing stepsizes of the form $\gamma_k=0.1/\sqrt{k}$, $k=1,2,\ldots$ \label{fig:dpcp sublinear}}
	\end{subfigure}
\hfill
	\begin{subfigure}{0.48\linewidth}
		\centerline{
			\includegraphics[width=2.5in]{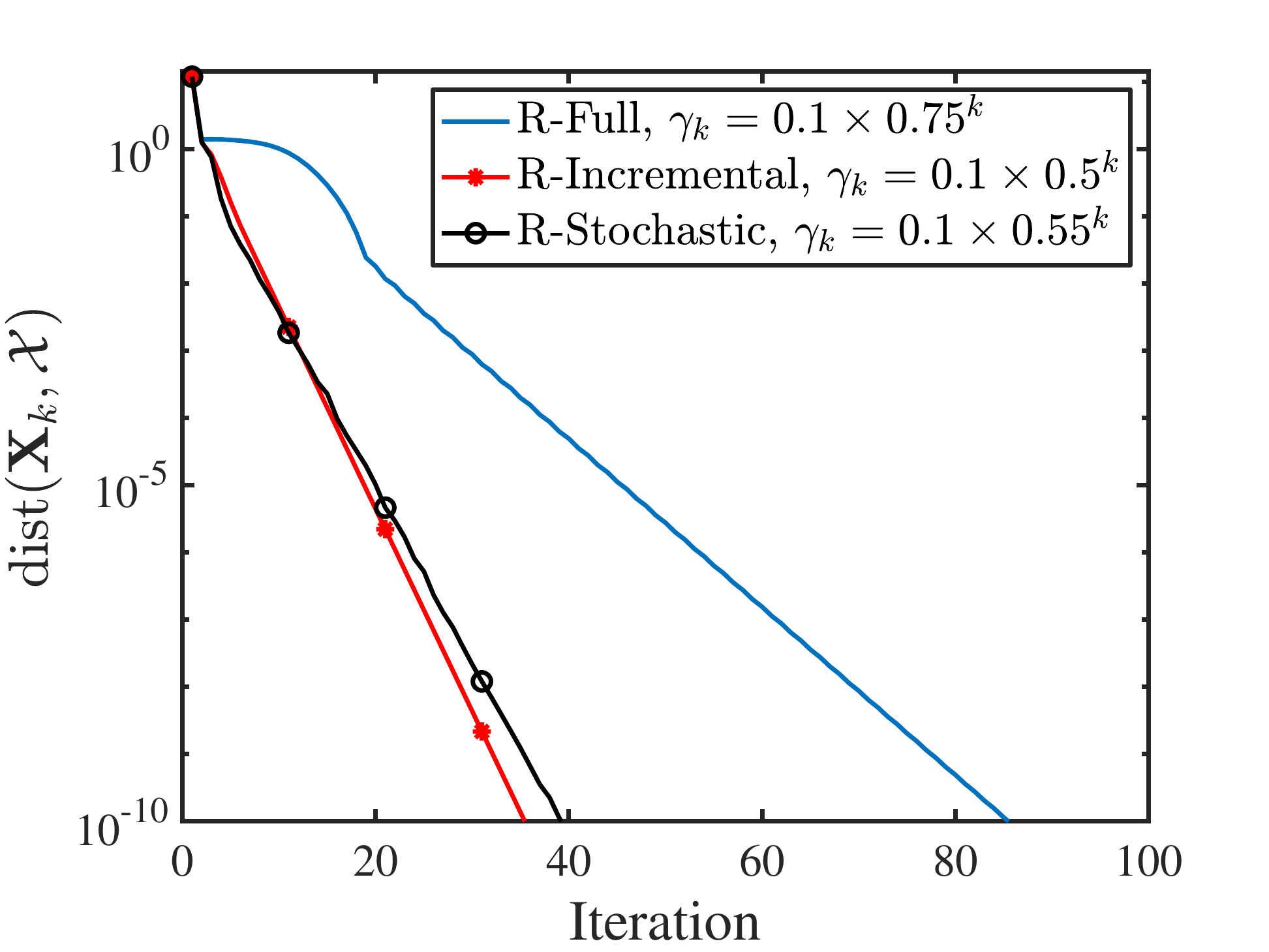}}
		\caption{\footnotesize DPCP: Geometrically diminishing stepsizes of the form $\gamma_k=0.1\times\beta^k$, $k=0,1,\ldots$ \label{fig:dpcp linear}}
	\end{subfigure}
	\caption{\small  Convergence performance of Riemannian subgradient-type methods for the DPCP formulation \eqref{eq:dpcp stiefel}.
	}\label{fig:dpcp experiments}
\end{figure}
We first randomly sample a subspace $\calS$ with co-dimension $r=10$  in ambient dimension $n = 100$. We then generate $m_1 = 1500$ inliers uniformly at random from the unit sphere in  $\calS$ and $m_2 = 3500$  outliers uniformly at random from the unit sphere in  $\R^n$.  We generate a standard Gaussian random vector and use it to initialize all the algorithms, as such an initialization provides comparable performance with the carefully designed initialization in \cite{maunu2019well,zhu2018dual}.  The numerical results are displayed in \Cref{fig:dpcp experiments}.
Sublinear convergence can be observed from the log-log plot in \Cref{fig:dpcp sublinear}, where we use the diminishing stepsizes suggested in \Cref{thm:rigd global rate} and \Cref{thm:stochastic subgradient global rate}. 
In \Cref{fig:dpcp linear}, we use geometrically diminishing stepsizes of the form $\gamma_k = \beta^k\gamma_0$. We fix $\gamma_0 = 0.1$ and tune the best decay factor $\beta$ for each algorithm. A linear rate of convergence can be observed, which corroborates our theoretical results.

\subsection{Orthogonal dictionary learning (ODL)}
We now turn to the orthogonal DL problem. Given $\mY = \mA\mS\in \R^{n\times m}$, where $\mA\in \stiefel(n,n)$ is an unknown orthonormal dictionary and each column of $\mS\in \R^{n\times m}$ is sparse, we can try to recover the columns of $\mA$ one at a time by considering the formulation~\eqref{eq:ODL sphere}, whose objective function takes the form $ \setS^{n-1} \ni \vx \mapsto f(\vx) =  \frac{1}{m}  \sum_{i= 1}^{m} \left| \vy_i^\top \vx \right|$, or to recover the entire dictionary by considering the formulation~\eqref{eq:ODL orthogonal}, whose objective function takes the form $ \stiefel(n,n) \ni \mX \mapsto f(\mX) =  \frac{1}{m}  \sum_{i= 1}^{m} \left\| \vy_i^\top \mX \right\|_1$.

\paragraph{Sharpness} The sharpness property of the formulation \eqref{eq:ODL sphere} has been studied in \cite{bai2018subgradient}, while that of \eqref{eq:ODL orthogonal} has been studied in \cite{wang2019unique} only in the \emph{asymptotic} regime; i.e., when the number of samples $m$ tends to infinity. Although we do not yet know how to establish the sharpness property of \eqref{eq:ODL orthogonal} in the \emph{finite-sample} regime, the following numerical results suggest that problem \eqref{eq:ODL orthogonal} likely possesses such a property, as the Riemannian subgradient-type methods with geometrically diminishing stepsizes exhibit linear convergence behavior, even with a random initialization. We leave the study of the sharpness property of~\eqref{eq:ODL orthogonal} in the finite-sample regime as a future work. 

\paragraph{Experiments}

\begin{figure}[t]
	\begin{subfigure}{0.48\linewidth}
		\centerline{
			\includegraphics[width=2.5in]{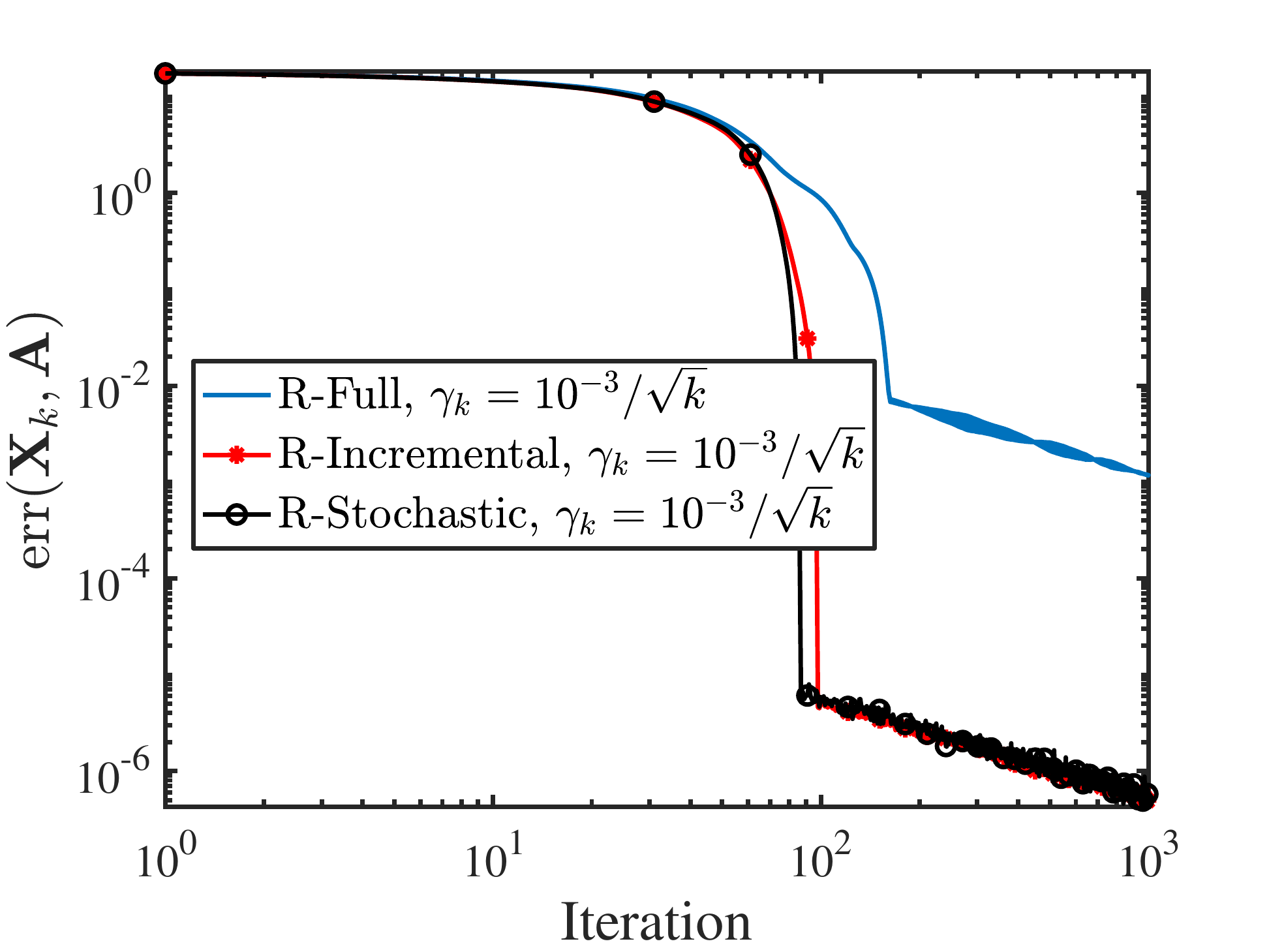}}
		\caption{\footnotesize ODL: Diminishing stepsizes of the form $\gamma_k=10^{-3}/\sqrt{k}$, $k=1,2,\ldots$ \label{fig:odl sublinear}}
	\end{subfigure}
\hfill
	\begin{subfigure}{0.48\linewidth}
		\centerline{
			\includegraphics[width=2.5in]{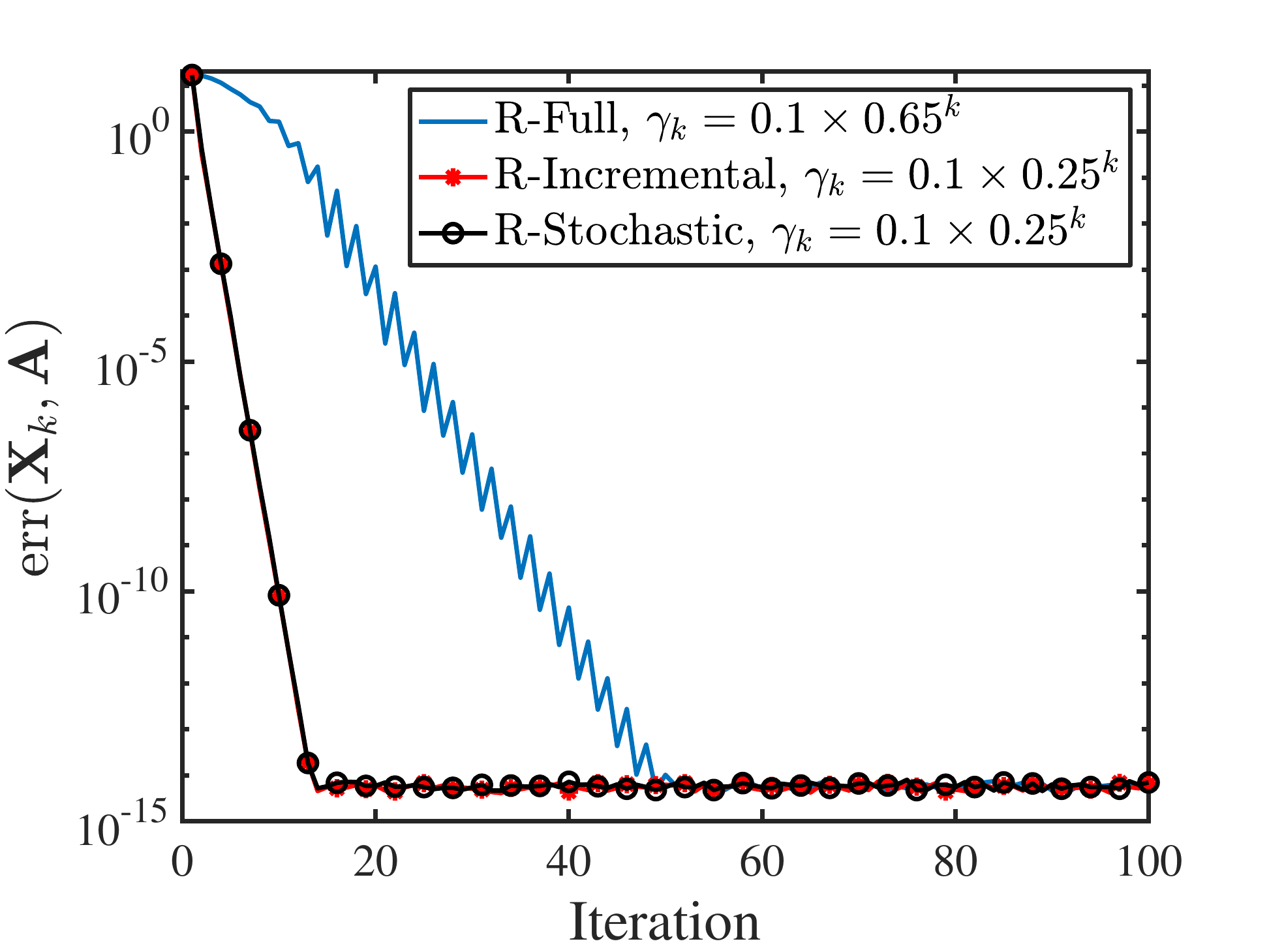}}
		\caption{\footnotesize ODL: Geometrically diminishing stepsizes of the form $\gamma_k=0.1\times\beta^k$, $k=0,1,\ldots$  \label{fig:odl linear}}
	\end{subfigure}
	\caption{\small Convergence performance of Riemannian subgradient-type methods for the orthogonal DL problem \eqref{eq:ODL orthogonal}.
	}\label{fig:odl experiments}
\end{figure}

For the orthogonal DL application, we generate synthetic data in the same way as \cite{bai2018subgradient}.  Specifically, we first generate the underlying orthogonal dictionary $\mA\in \stiefel(n,n)$ with $n=30$ randomly and set the number of samples $m$ to be $m = 1643 \approx 10 \times n^{1.5}$.  We then generate a sparse coefficient matrix $\mS\in \R^{n\times m}$, in which each entry follows the Bernoulli-Gaussian distribution with parameter $0.3$ (sparsity)---i.e., each entry $\mS_{i,j}$ is drawn independently from the standard Gaussian distribution with probability $0.3$ and is set to zero otherwise. Lastly, we obtain the observation $\mY = \mA\mS$.  As before, we generate a standard Gaussian random vector and use it to initialize all the algorithms.  To evaluate the performance of the algorithms, we define the error between $\mX$ and $\mA$ as ${\rm err}(\mX,\mA) = \sum_{i=1}^{n}  \left|  \max_{1\leq j \leq n} \left|  [\vx_i^\top \mA]_{j} \right| - 1 \right|$, where $\vx_i$ is the $i$-th column of $\mX$. Clearly, ${\rm err}(\mX,\mA)=0$ when $\mX$ and $\mA$ are equal up to permutation and sign ambiguities. The numerical results are shown in \Cref{fig:odl experiments}.  The log-log plot in \Cref{fig:odl sublinear} shows the sublinear convergence of Riemannian subgradient-type methods when the diminishing stepsizes suggested in \Cref{thm:rigd global rate} and \Cref{thm:stochastic subgradient global rate} are used. 
\Cref{fig:odl linear} shows the linear convergence of those methods when geometrically diminishing stepsizes of the form $\gamma_k = \beta^k\gamma_0$ are used. Here, $\gamma_0 = 0.1$ and the best decay factor $\beta$ is chosen for each algorithm.

\section{Conclusion}

In this work, we introduced a family of Riemannian subgradient-type methods for minimizing weakly convex functions over the Stiefel manifold. We proved, for the first time, iteration complexity and local convergence rate results for these methods. Specifically, we showed that all these methods have a global sublinear convergence rate, and that if the problem at hand further possesses the sharpness property, then the Riemannian subgradient and incremental subgradient methods with geometrically diminishing stepsizes and a proper initialization will converge linearly to the set of weak sharp minima of the problem. The key to establishing these results is a new Riemannian subgradient inequality for restrictions of weakly convex functions on the Stiefel manifold, which could be of independent interest. Our results can be extended to cover weakly convex minimization over a class of compact embedded submanifolds of the Euclidean space. Lastly, we showed that certain formulations of the RSR and orthogonal DL problems possess the sharpness property and verified the convergence performance of the Riemannian subgradient-type methods on these problems via numerical simulations. 

Our work has opened up several interesting directions for future investigation. First, one can readily generalize our results to weakly convex minimization over a Cartesian product of Stiefel manifolds, which has applications in $\ell_1$-PCA \cite{lerman2018overview,WLS19} and robust phase synchronization \cite{wang2013exact}. Next, since our results are specific to weakly convex minimization over the Stiefel manifold, it would be interesting to see if they can be extended to handle more general nonconvex nonsmooth functions over a broader class of Riemannian manifolds. We believe that this should be possible based on the analytic framework developed here. Finally, we suspect that the global convergence rate $\calO(k^{-\frac{1}{4}})$ we established for the Riemannian subgradient-type methods is not tight. This is because the Riemannian proximal point method for solving problem \eqref{eq:stiefel opt problem} has a global convergence rate of $\calO(k^{-\frac{1}{2}})$ \cite{Chen2020manppa}, and in smooth optimization the gradient descent method has the same global convergence rate as the proximal point method. Hence, it would be interesting to see if the global convergence rate established in this paper can be improved.

\section*{Acknowledgments} We would like to thank Dr. Huikang Liu for fruitful discussions. We also thank the Associate Editor and two anonymous reviewers for their detailed and helpful comments.

\bibliographystyle{siamplain}
\bibliography{nonconvex,manifold}

\begin{thebibliography}{10}

\bibitem{absil2019collection}
{\sc P.-A. Absil and S.~Hosseini}, {\em A collection of nonsmooth {R}iemannian
  optimization problems}, in Nonsmooth Optimization and Its Applications,
  Springer, 2019, pp.~1--15.

\bibitem{absil2009optimization}
{\sc P.-A. Absil, R.~Mahony, and R.~Sepulchre}, {\em Optimization {A}lgorithms
  on {M}atrix {M}anifolds}, Princeton University Press, 2009.

\bibitem{bai2018subgradient}
{\sc Y.~Bai, Q.~Jiang, and J.~Sun}, {\em Subgradient descent learns orthogonal
  dictionaries}, in International Conference on Learning Representations, 2019.

\bibitem{bento2017iteration}
{\sc G.~C. Bento, O.~P. Ferreira, and J.~G. Melo}, {\em Iteration-complexity of
  gradient, subgradient and proximal point methods on {R}iemannian manifolds},
  Journal of Optimization Theory and Applications, 173 (2017), pp.~548--562.

\bibitem{bishop1969manifolds}
{\sc R.~L. Bishop and B.~O'Neill}, {\em Manifolds of negative curvature},
  Transactions of the American Mathematical Society, 145 (1969), pp.~1--49.

\bibitem{boucheron2013concentration}
{\sc S.~Boucheron, G.~Lugosi, and P.~Massart}, {\em Concentration Inequalities:
  A Nonasymptotic Theory of Independence}, Oxford University Press, 2013.

\bibitem{boumal2018global}
{\sc N.~Boumal, P.-A. Absil, and C.~Cartis}, {\em Global rates of convergence
  for nonconvex optimization on manifolds}, IMA Journal of Numerical Analysis,
  39 (2019), pp.~1--33.

\bibitem{burke1993weak}
{\sc J.~V. Burke and M.~C. Ferris}, {\em Weak sharp minima in mathematical
  programming}, SIAM Journal on Control and Optimization, 31 (1993),
  pp.~1340--1359.

\bibitem{burke2005robust}
{\sc J.~V. Burke, A.~S. Lewis, and M.~L. Overton}, {\em A robust gradient
  sampling algorithm for nonsmooth, nonconvex optimization}, SIAM Journal on
  Optimization, 15 (2005), pp.~751--779.

\bibitem{Chen2020manppa}
{\sc S.~Chen, Z.~Deng, S.~Ma, and A.~M.-C. So}, {\em Manifold proximal point
  algorithms for dual principal component pursuit and orthogonal dictionary
  learning}, arXiv preprint arXiv:2005.02356,  (2020).

\bibitem{chen1811proximal}
{\sc S.~Chen, S.~Ma, A.~M.-C. So, and T.~Zhang}, {\em Proximal gradient method
  for nonsmooth optimization over the {S}tiefel manifold}, SIAM Journal on
  Optimization, 30 (2020), pp.~210--239.

\bibitem{davis2019stochasticmodel}
{\sc D.~Davis and D.~Drusvyatskiy}, {\em Stochastic model-based minimization of
  weakly convex functions}, SIAM Journal on Optimization, 29 (2019),
  pp.~207--239.

\bibitem{davis2018subgradient}
{\sc D.~Davis, D.~Drusvyatskiy, K.~J. MacPhee, and C.~Paquette}, {\em
  Subgradient methods for sharp weakly convex functions}, Journal of
  Optimization Theory and Applications, 179 (2018), pp.~962--982.

\bibitem{de2016new}
{\sc G.~de~Carvalho~Bento, J.~X. da~Cruz~Neto, and P.~R. Oliveira}, {\em A new
  approach to the proximal point method: Convergence on general {R}iemannian
  manifolds}, Journal of Optimization Theory and Applications, 168 (2016),
  pp.~743--755.

\bibitem{drusvyatskiy2017proximal}
{\sc D.~Drusvyatskiy}, {\em The proximal point method revisited}, SIAG/OPT
  Views and News, 26 (2018), pp.~1--7.

\bibitem{drusvyatskiy2018efficiency}
{\sc D.~Drusvyatskiy and C.~Paquette}, {\em Efficiency of minimizing
  compositions of convex functions and smooth maps}, Mathematical Programming,
  178 (2019), pp.~503--558.

\bibitem{edelman1998geometry}
{\sc A.~Edelman, T.~A. Arias, and S.~T. Smith}, {\em The geometry of algorithms
  with orthogonality constraints}, SIAM Journal on Matrix Analysis and
  Applications, 20 (1998), pp.~303--353.

\bibitem{ferreira2019iteration}
{\sc O.~Ferreira, M.~Louzeiro, and L.~Prudente}, {\em Iteration-complexity of
  the subgradient method on {R}iemannian manifolds with lower bounded
  curvature}, Optimization, 68 (2019), pp.~713--729.

\bibitem{ferreira1998subgradient}
{\sc O.~P. Ferreira and P.~R. Oliveira}, {\em Subgradient algorithm on
  {R}iemannian manifolds}, Journal of Optimization Theory and Applications, 97
  (1998), pp.~93--104.

\bibitem{Ferreira-Oliveira-PPA-Manifold-2002}
{\sc O.~P. Ferreira and \vspace{0mm} P.~R.~Oliveira}, {\em Proximal point
  algorithm on {R}iemannian manifold}, Optimization, 51 (2002), pp.~257--270.

\bibitem{goffin1977convergence}
{\sc J.-L. Goffin}, {\em On convergence rates of subgradient optimization
  methods}, Mathematical Programming, 13 (1977), pp.~329--347.

\bibitem{hosseini2018line}
{\sc S.~Hosseini, W.~Huang, and R.~Yousefpour}, {\em Line search algorithms for
  locally {L}ipschitz functions on {R}iemannian manifolds}, SIAM Journal on
  Optimization, 28 (2018), pp.~596--619.

\bibitem{Hosseini-Uschmajew-2017}
{\sc S.~Hosseini and A.~Uschmajew}, {\em A {R}iemannian gradient sampling
  algorithm for nonsmooth optimization on manifolds}, SIAM Journal on
  Optimization., 27 (2017), pp.~173--189.

\bibitem{hu2019brief}
{\sc J.~Hu, X.~Liu, Z.-W. Wen, and Y.-X. Yuan}, {\em A brief introduction to
  manifold optimization}, Journal of the Operations Research Society of China,
  8 (2020), pp.~199--248.

\bibitem{huang2019riemannian}
{\sc W.~Huang and K.~\vspace{0mm} Wei}, {\em Riemannian proximal gradient
  methods}, arXiv preprint arXiv:1909.06065,  (2019).

\bibitem{jiang2017vector}
{\sc B.~Jiang, S.~Ma, A.~M.-C. So, and S.~Zhang}, {\em Vector transport-free
  {SVRG} with general retraction for {R}iemannian optimization: Complexity
  analysis and practical implementation}, arXiv preprint arXiv:1705.09059,
  (2017).

\bibitem{Journee-Nesterov-sparsePCA-JMLR-2010}
{\sc M.~Journ\'{e}e, Y.~Nesterov, P.~Richt\'{a}rik, and R.~Sepulchre}, {\em
  Generalized power method for sparse principal component analysis}, Journal of
  Machine Learning Research, 11 (2010), pp.~517--553.

\bibitem{karkhaneei2019nonconvex}
{\sc M.~M. Karkhaneei and N.~Mahdavi-Amiri}, {\em Nonconvex weak sharp minima
  on {R}iemannian manifolds}, Journal of Optimization Theory and Applications,
  183 (2019), pp.~85--104.

\bibitem{knyazev2007majorization}
{\sc A.~V. Knyazev and M.~E. Argentati}, {\em Majorization for changes in
  angles between subspaces, {R}itz values, and graph {L}aplacian spectra}, SIAM
  Journal on Matrix Analysis and Applications, 29 (2007), pp.~15--32.

\bibitem{Kovnatsky2016}
{\sc A.~Kovnatsky, K.~Glashoff, and M.~M. Bronstein}, {\em {MADMM}: a generic
  algorithm for non-smooth optimization on manifolds}, in European Conference
  on Computer Vision, Springer, 2016, pp.~680--696.

\bibitem{Lai2014}
{\sc R.~Lai and S.~Osher}, {\em A splitting method for orthogonality
  constrained problems}, Journal of Scientific Computing, 58 (2014),
  pp.~431--449.

\bibitem{lerman2018fast}
{\sc G.~Lerman and T.~Maunu}, {\em Fast, robust and non-convex subspace
  recovery}, Information and Inference: A Journal of the IMA, 7 (2018),
  pp.~277--336.

\bibitem{lerman2018overview}
{\sc G.~Lerman and T.~Maunu}, {\em An overview of robust subspace recovery},
  Proceedings of the IEEE, 106 (2018), pp.~1380--1410.

\bibitem{lerman2015robust}
{\sc G.~Lerman, M.~B. McCoy, J.~A. Tropp, and T.~Zhang}, {\em Robust
  computation of linear models by convex relaxation}, Foundations of
  Computational Mathematics, 15 (2015), pp.~363--410.

\bibitem{li2011weak}
{\sc C.~Li, B.~S. Mordukhovich, J.~Wang, and J.-C. Yao}, {\em Weak sharp minima
  on {R}iemannian manifolds}, SIAM Journal on Optimization, 21 (2011),
  pp.~1523--1560.

\bibitem{li2019incremental}
{\sc X.~Li, Z.~Zhu, A.~M.-C. So, and J.~D. Lee}, {\em Incremental methods for
  weakly convex optimization}, arXiv preprint arXiv:1907.11687,  (2019).

\bibitem{li2018nonconvex}
{\sc X.~Li, Z.~Zhu, A.~M.-C. So, and R.~Vidal}, {\em Nonconvex robust low-rank
  matrix recovery}, SIAM Journal on Optimization, 30 (2020), pp.~660--686.

\bibitem{Liu-So-Wu-2018}
{\sc H.~Liu, A.~M.-C. So, and W.~Wu}, {\em Quadratic optimization with
  orthogonality constraint: Explicit {{\L}}ojasiewicz exponent and linear
  convergence of retraction-based line-search and stochastic variance-reduced
  gradient methods}, Mathematical Programming, 178 (2019), pp.~215--262.

\bibitem{mairal2014sparse}
{\sc J.~Mairal, F.~Bach, and J.~Ponce}, {\em Sparse modeling for image and
  vision processing}, Foundations and Trends\textsuperscript{\textregistered}
  in Computer Graphics and Vision, 8 (2014), pp.~85--283.

\bibitem{maunu2019well}
{\sc T.~Maunu, T.~Zhang, and G.~Lerman}, {\em A well-tempered landscape for
  non-convex robust subspace recovery.}, Journal of Machine Learning Research,
  20 (2019), pp.~1--59.

\bibitem{maurer2016vector}
{\sc A.~Maurer}, {\em A vector-contraction inequality for {R}ademacher
  complexities}, in Proceedings of the 27th International Conference on
  Algorithmic Learning Theory (ALT 2016), R.~Ortner, H.~U. Simon, and
  S.~Zilles, eds., vol.~9925 of Lecture Notes in Artificial Intelligence, 2016,
  pp.~3--17.

\bibitem{nedic2001convergence}
{\sc A.~Nedi\'{c} and D.~Bertsekas}, {\em Convergence rate of incremental
  subgradient algorithms}, in Stochastic Optimization: Algorithms and
  Applications, S.~Uryasev and P.~M. Pardalos, eds., vol.~54 of Applied
  Optimization, Springer Science+Business Media, Dordrecht, 2001, pp.~223--264.

\bibitem{nemirovski2009robust}
{\sc A.~Nemirovski, A.~Juditsky, G.~Lan, and A.~Shapiro}, {\em Robust
  stochastic approximation approach to stochastic programming}, SIAM Journal on
  Optimization, 19 (2009), pp.~1574--1609.

\bibitem{Nesterov:2014:ILC:2670022}
{\sc {\relax Yu}.~Nesterov}, {\em {Introductory Lectures on Convex
  Optimization: A Basic Course}}, Kluwer Academic Publishers, Boston, 2004.

\bibitem{qu2016finding}
{\sc Q.~Qu, J.~Sun, and J.~Wright}, {\em Finding a sparse vector in a subspace:
  Linear sparsity using alternating directions}, IEEE Transactions on
  Information Theory, 62 (2016), pp.~5855--5880.

\bibitem{rockafellar2009variational}
{\sc R.~T. Rockafellar and R.~J.-B. Wets}, {\em Variational Analysis}, vol.~317
  of Grundlehren der mathematischen Wissenschaften, Springer Science \&
  Business Media, second~ed., 2009.

\bibitem{rubinstein2010dictionaries}
{\sc R.~Rubinstein, A.~M. Bruckstein, and M.~Elad}, {\em Dictionaries for
  sparse representation modeling}, Proceedings of the IEEE, 98 (2010),
  pp.~1045--1057.

\bibitem{S66}
{\sc P.~H. Sch\"{o}nemann}, {\em A generalized solution of the orthogonal
  procrustes problem}, Psychometrika, 31 (1966), pp.~1--10.

\bibitem{Shor:1985:MMN:3585}
{\sc N.~Z. Shor}, {\em Minimization Methods for Non-Differentiable Functions},
  vol.~3 of Springer Series in Computational Mathematics, Springer--Verlag,
  Berlin Heidelberg, 1985.

\bibitem{spielman2012exact}
{\sc D.~A. Spielman, H.~Wang, and J.~Wright}, {\em Exact recovery of
  sparsely-used dictionaries}, in Proceedings of the 25th Annual Conference on
  Learning Theory, 2012, pp.~37.1--37.18.

\bibitem{stewart1990matrix}
{\sc G.~W. Stewart and J.~Sun}, {\em {Matrix Perturbation Theory}}, Academic
  Press, Boston, 1990.

\bibitem{sun2016complete_a}
{\sc J.~Sun, Q.~Qu, and J.~Wright}, {\em Complete dictionary recovery over the
  sphere {I}: Overview and the geometric picture}, IEEE Transactions on
  Information Theory, 63 (2016), pp.~853--884.

\bibitem{sun2016complete_b}
{\sc J.~Sun, Q.~Qu, and J.~Wright}, {\em Complete dictionary recovery over the
  sphere {II}: Recovery by {R}iemannian trust-region method}, IEEE Transactions
  on Information Theory, 63 (2016), pp.~885--914.

\bibitem{tsakiris2018dual}
{\sc M.~C. Tsakiris and R.~Vidal}, {\em Dual principal component pursuit},
  Journal of Machine Learning Research, 19 (2018), pp.~1--49.

\bibitem{V83}
{\sc J.-P. Vial}, {\em Strong and weak convexity of sets and functions},
  Mathematics of Operations Research, 8 (1983), pp.~231--259.

\bibitem{ma2015generalized}
{\sc R.~Vidal, Y.~Ma, and S.~S. Sastry}, {\em Generalized Principal Component
  Analysis}, vol.~40 of Interdisciplinary Applied Mathematics, Springer-Verlag,
  New York, 2016.

\bibitem{wang2013exact}
{\sc L.~Wang and A.~Singer}, {\em Exact and stable recovery of rotations for
  robust synchronization}, Information and Inference: A Journal of the IMA, 2
  (2013), pp.~145--193.

\bibitem{WLS19}
{\sc P.~Wang, H.~Liu, and A.~M.-C. So}, {\em Globally convergent accelerated
  proximal alternating maximization method for {L1}--principal component
  analysis}, in Proceedings of the 2019 IEEE International Conference on
  Acoustics, Speech, and Signal Processing (ICASSP 2019), 2019, pp.~8147--8151.

\bibitem{wang2019unique}
{\sc Y.~Wang, S.~Wu, and B.~Yu}, {\em Unique sharp local minimum in
  $\ell_1$-minimization complete dictionary learning}, Journal of Machine
  Learning Research, 21 (2020), pp.~1--52.

\bibitem{wen2013feasible}
{\sc Z.~Wen and W.~Yin}, {\em A feasible method for optimization with
  orthogonality constraints}, Mathematical Programming, 142 (2013),
  pp.~397--434.

\bibitem{wright2010sparse}
{\sc J.~Wright, Y.~Ma, J.~Mairal, G.~Sapiro, T.~S. Huang, and S.~Yan}, {\em
  Sparse representation for computer vision and pattern recognition},
  Proceedings of the IEEE, 98 (2010), pp.~1031--1044.

\bibitem{xu2012outlier}
{\sc H.~Xu, C.~Caramanis, and S.~Mannor}, {\em Outlier-robust {PCA}: The
  high-dimensional case}, IEEE Transactions on Information Theory, 59 (2012),
  pp.~546--572.

\bibitem{yang2014optimality}
{\sc W.~H. Yang, L.-H. Zhang, and R.~Song}, {\em Optimality conditions for the
  nonlinear programming problems on {R}iemannian manifolds}, Pacific Journal of
  Optimization, 10 (2014), pp.~415--434.

\bibitem{yau1974non}
{\sc S.-T. Yau}, {\em Non-existence of continuous convex functions on certain
  {R}iemannian manifolds}, Mathematische Annalen, 207 (1974), pp.~269--270.

\bibitem{zhai2019complete}
{\sc Y.~Zhai, Z.~Yang, Z.~Liao, J.~Wright, and Y.~Ma}, {\em Complete dictionary
  learning via $\ell^4$-norm maximization over the orthogonal group}, Journal
  of Machine Learning Research, 21 (2020), pp.~1--68.

\bibitem{Sra-1st-order-geodesically-convex-2016}
{\sc H.~Zhang and S.~Sra}, {\em First-order methods for geodesically convex
  optimization}, in Proceedings of the 29th Annual Conference on Learning
  Theory, 2016, pp.~1617--1638.

\bibitem{zhang2016robust}
{\sc T.~Zhang}, {\em Robust subspace recovery by {T}yler's {M}-estimator},
  Information and Inference: A Journal of the IMA, 5 (2016), pp.~1--21.

\bibitem{zhang2014novel}
{\sc T.~Zhang and G.~Lerman}, {\em A novel {M}-estimator for robust {PCA}},
  Journal of Machine Learning Research, 15 (2014), pp.~749--808.

\bibitem{zhu2019grasssub}
{\sc Z.~Zhu, T.~Ding, M.~Tsakiris, D.~Robinson, and R.~Vidal}, {\em A linearly
  convergent method for non-smooth non-convex optimization on {G}rassmannian
  with applications to robust subspace and dictionary learning}, in Advances in
  Neural Information Processing Systems, 2019, pp.~9437--9447.

\bibitem{zhu2018dual}
{\sc Z.~Zhu, Y.~Wang, D.~Robinson, D.~Naiman, R.~Vidal, and M.~Tsakiris}, {\em
  Dual principal component pursuit: Improved analysis and efficient
  algorithms}, in Advances in Neural Information Processing Systems, 2018,
  pp.~2171--2181.

\end{thebibliography}

\appendix
\section{Proof of \Cref{lem:statistics of outliers and inliers}}
\label{sec:appendix}

The proof follows the framework in \cite[Section 8.1.1]{lerman2015robust} with nontrivial modifications in order to handle our matrix-based definitions of $c_{\mY,\min}$ and $c_{\mO,\max}$.

\emph{Part I.} We first derive an upper bound on $c_{\mO,\max}$. Recall that under the Haystack model, the outliers $\vo_1,\ldots, \vo_{m_2} \in \R^n$ are i.i.d. according to the Gaussian distribution $\calN(\mathbf{0},\frac{1}{n}\mId_n)$.  Let $\vo\sim\calN(\mathbf{0},\frac{1}{n}\mId_n)$ denote an i.i.d. copy of $\vo_i$. Then, we have
	\e\label{eq:split} 
	\begin{split}
	     & \sup_{\|\mB\|_F = 1} \ \sum_{i = 1}^{m_2} \|\vo_i^\top \mB\|_2   \\
	     \leq&\sup_{\|\mB\|_F = 1} \ \sum_{i = 1}^{m_2} \left( \|\vo_i^\top \mB\|_2 - \E\left[\|\vo^\top \mB\|_2\right] \right) +  \sup_{\|\mB\|_F = 1} \ \sum_{i = 1}^{m_2} \ \E\left[\|\vo^\top \mB\|_2\right].
	 \end{split}
	\ee
Using Jensen's inequality, we bound the second term as follows:
\e \label{eq:expectation}
\begin{split}
    & \sup_{\|\mB\|_F = 1} \ \sum_{i = 1}^{m_2} \ \E\left[\|\vo^\top \mB\|_2\right]  \leq \sup_{\|\mB\|_F = 1} \ \sum_{i = 1}^{m_2} \ \sqrt{ \E\left[\|\vo^\top \mB\|_2^2\right] } \\
    &= \sup_{\|\mB\|_F = 1} \ \sum_{i = 1}^{m_2} \ \sqrt{ \sum_{j=1}^{r} \sum_{k=1}^n \E\left[ o_{k}^2 b_{kj}^2\right] } = \frac{m_2}{\sqrt{n}}.
\end{split}
\ee
To estimate the first term in~\eqref{eq:split}, let $\{\epsilon_i : i =1,\ldots, m_2\}$ be independent Rademacher random variables (i.e., $\P{\epsilon_i = +1} = \P{\epsilon_i = -1} = 1/2$ for $i=1,\ldots,m_2$) that are independent of $\{\vo_i:i=1,\ldots,m_2\}$. By a standard symmetrization argument (see, e.g., \cite[Lemma 11.4]{boucheron2013concentration}), we have
\e\label{eq:Rademacher symmetrization}
	     \E\left[ \sup_{\|\mB\|_F = 1} \ \sum_{i = 1}^{m_2} \left( \|\vo_i^\top \mB\|_2 - \E\left[\|\vo^\top \mB\|_2\right] \right) \right] \leq 2 \E \left[    \sup_{\|\mB\|_F = 1} \ \sum_{i = 1}^{m_2} \epsilon_i \|\vo_i^\top \mB\|_2  \right].
\ee
Furthermore, let $\{\epsilon_{ij} : i=1,\ldots,m_2; \, j=1,\ldots,r\}$ be independent Rademacher random variables that are independent of $\{\vo_i:i=1,\ldots,m_2\}$ and $\mV \in \R^{n\times r}$ be the matrix whose $j$-th column ($j=1,\ldots,r$) is $\sum_{i = 1}^{m_2} \epsilon_{ij} \vo_i$. Then, by the vector contraction inequality in \cite[Corollary 1]{maurer2016vector} and Jensen's inequality, we have
\begin{align*}
	  &\E \left[    \sup_{\|\mB\|_F = 1} \ \sum_{i = 1}^{m_2} \epsilon_i \|\vo_i^\top \mB\|_2  \right]  \leq \sqrt{2} \E \left[ \sup_{\|\mB\|_F = 1}  \ \sum_{i = 1}^{m_2} \sum_{j=1}^r \epsilon_{ij}  \vo_i^\top  \vb_j   \right] \\
	  & =  \sqrt{2} \E \left[ \sup_{\|\mB\|_F = 1}  \  \sum_{j=1}^r   \left( \sum_{i = 1}^{m_2} \epsilon_{ij} \vo_i \right)^\top \vb_j  \right] 
	   =  \sqrt{2} \E \left[ \sup_{\|\mB\|_F = 1}    \left\langle \mV,   \mB \right\rangle  \right] = \sqrt{2} \E \left[ \| \mV\|_F \right] \\
	  & \leq   \sqrt{2} \sqrt{  \sum_{j=1}^r   \E \left[ \left\| \sum_{i = 1}^{m_2} \epsilon_{ij} \vo_i\right\|_2^2  \right] } 
	  \leq \sqrt{2m_2r}.
\end{align*}
This, together with \eqref{eq:Rademacher symmetrization}, yields
\e \label{eq:expectation bound}
    \E\left[ \sup_{\|\mB\|_F = 1} \ \sum_{i = 1}^{m_2} \left( \|\vo_i^\top \mB\|_2 - \E\left[\|\vo^\top \mB\|_2\right] \right) \right] \leq 2\sqrt{2m_2r}. 
\ee
Now, observe that the function 
\[ (\vo_1,\ldots, \vo_{m_2}) \mapsto h(\vo_1,\ldots, \vo_{m_2}) := \sup_{\|\mB\|_F = 1} \ \sum_{i = 1}^{m_2} \left( \|\vo_i^\top \mB\|_2 - \E\left[\|\vo^\top \mB\|_2\right] \right) \]
is Lipschitz continuous with constant at most $\sqrt{m_2}$. 
Hence, using the Gaussian concentration inequality for Lipschitz functions \cite[Theorem 5.6]{boucheron2013concentration} and \eqref{eq:expectation bound}, we get 
\e \label{eq:high probability bound}
    \P{ \sup_{\|\mB\|_F = 1} \ \sum_{i = 1}^{m_2} \left( \|\vo_i^\top \mB\|_2 - \E\left[\|\vo^\top \mB\|_2\right] \right) \leq 2\sqrt{2m_2r}  + t } \geq 1- 2\exp\left( -\frac{nt^2}{2m_2} \right).
\ee
Upon substituting \eqref{eq:high probability bound} and \eqref{eq:expectation} into \eqref{eq:split} and letting $t = c_2\sqrt{m_2}$, the desired result follows.

\smallskip
\emph{Part II.}  We now derive a lower bound on $c_{\mY,\min}$. Again, recall that under the Haystack model, the inliers $\vy_1,\ldots, \vy_{m_1} \in \R^n$ are i.i.d. according to the Gaussian distribution $\calN(\mathbf{0},\frac{1}{d}\calP_{\calS})$. Thus, for $i=1,\ldots,m_1$, we have $\E[\vy_i\vy_i^\top] = \frac{1}{d}\calP_{\calS} = \frac{1}{d}\mS\mS^\top$ for some orthonormal basis $\mS\in\stiefel(n,d)$ of $\calS$ and $\vy_i = \mS\widetilde{\vy}_i$ for some $\widetilde{\vy}_i \in \R^d$. Now, let $\mD \in \R^{n\times \ell}$ be such that $\|\mD\|_F=1$ and ${\rm col}(\mD) \subseteq \calS$; see~\eqref{eq:cXmin}. Then, there exists a $\widetilde{\mD} \in \R^{d \times \ell}$ such that $\mD = \mS\widetilde{\mD}$ and $\|\widetilde{\mD}\|_F=1$. In particular, we have $\vy_i^\top\mD = \widetilde{\vy}_i^\top\widetilde{\mD}$, and by the rotational invariance of the Gaussian distribution, the vector $\widetilde{\vy}_i$ follows the Gaussian distribution $\calN(\bm{0},\frac{1}{d}\mId_d)$ in $\R^d$. Consequently, we may assume without loss of generality that $\vy_1,\ldots,\vy_{m_1} \in \R^d$ are i.i.d. according to the Gaussian distribution $\calN(\mathbf{0},\frac{1}{d}\mId_d)$ and $\mD \in \R^{d\times \ell}$ satisfies $\|\mD\|_F = 1$. The rest of the proof will be similar to that of Part I. 


Let $\vy \sim \calN(\mathbf{0},\frac{1}{d}\mId_d)$ denote an i.i.d. copy of $\vy_i$. Then, we have 
\e\label{eq:split 2} 
\begin{split}
	& \inf_{\|\mD\|_F = 1} \ \sum_{i = 1}^{m_1} \|\vy_i^\top \mD\|_2   \\
	&\quad \geq \inf_{\|\mD\|_F = 1} \ \sum_{i = 1}^{m_1} \left( \|\vy_i^\top \mD\|_2 - \E\left[\|\vy^\top \mD\|_2\right] \right) +  \inf_{\|\mD\|_F = 1} \ \sum_{i = 1}^{m_1} \ \E\left[\|\vy^\top \mD\|_2\right].
\end{split}
\ee
The first term can be written as
$
-\sup_{\|\mD\|_F = 1} \ \sum_{i = 1}^{m_1} \left( \E\left[\|\vy^\top \mD\|_2\right] - \|\vy_i^\top \mD\|_2  \right)
$.
By following the same arguments as in Part I, we obtain 
\e\label{eq:high probability bound 2}
   \P{ \sup_{\|\mD\|_F = 1} \ \sum_{i = 1}^{m_1} \left( \E\left[\|\vy^\top \mD\|_2\right] - \|\vy_i^\top \mD\|_2  \right) \leq 2\sqrt{2m_1 \ell} + t } \geq 1- 2\exp\left( -\frac{dt^2}{2m_1} \right).
\ee
It remains to estimate the second term in \eqref{eq:split 2}. Let $\vd_i$ be the $i$-th column of $\mD$, where $i=1,\ldots,\ell$. By the Cauchy-Schwarz inequality and the fact that $\|\mD\|_F = 1$, we have
\[
\begin{aligned}
& \underbrace{\left( \|\vd_1\|_2^2 + \cdots +\|\vd_{\ell}\|_2^2  \right)}_{= \|\mD\|_F^2 = 1} \underbrace{\left[ (\vy^\top \vd_1)^2 + \cdots +(\vy^\top \vd_{\ell})^2 \right]}_{=\|\vy^\top \mD\|_2^2} \\
&\qquad\geq \left[  \|\vd_1\|_2  | \vy^\top \vd_1| +\cdots+\|\vd_{\ell}\|_2  | \vy^\top \vd_{\ell}| \right]^2.
\end{aligned}
\]
Since $\vy^\top \vd_i \sim \calN(0,\frac{1}{d}\|\vd_i\|_2^2)$, we obtain
 \[
     \E\left[\|\vy^\top \mD\|_2\right] \geq \E\left[  \|\vd_1\|_2  | \vy^\top \vd_1| +\cdots+\|\vd_{\ell}\|_2  | \vy^\top \vd_{\ell}| \right] = \sqrt{\frac{2}{d\pi}}.
 \]
Note that the above inequality holds as equality when $\vd_1 = \cdots = \vd_{\ell}$. This implies that
\e\label{eq:expectation 2}
     \inf_{\|\mD\|_F = 1} \ \sum_{i = 1}^{m_1} \ \E\left[\|\vy^\top \mD\|_2\right] = m_1\sqrt{\frac{2}{d\pi}}.
\ee
By substituting \eqref{eq:high probability bound 2} with $t = c_1\sqrt{m_1}$ and \eqref{eq:expectation 2} into \eqref{eq:split 2}, we complete the proof.

\end{document}